\documentclass[11pt,letterpaper]{article}
\usepackage{epsfig}
\usepackage{amsfonts}
\usepackage{amssymb}
\usepackage{amsmath}
\usepackage{comment}
\usepackage{float}
\usepackage{verbatim}
\usepackage{graphics} 
\usepackage{times} 
\usepackage{amsbsy,bm,fixmath}
\usepackage[normalem]{ulem}
\usepackage{mathtools}

\def\n1LXWo{{L^2_{X_0(0);W_0}(0,T;\mathbb{R}^{n_1})}}

\def\mbA{\mathbold{A}}
\def\mbB{\mathbold{B}}
\def\mbD{\mathbold{D}}
\def\mbP{\mathbold{P}}
\def\mbx{\mathbold{x}}
\def\mbQ{\mathbold{Q}}
\def\mbS{\mathbold{S}}
\def\mbK{\mathbold{K}}

\newtheorem{theorem}{Theorem}
\newtheorem{definition}[theorem]{Definition}
\newtheorem{lemma}[theorem]{Lemma}

\newtheorem{remark}{Remark}
\def\proof{{\it Proof}.\quad}
\def\endproof{\hfill$\Box$}
\def\re{\eqref}

\begin{document}

\vspace{-0.3cm}
\title{\LARGE\bf Linear-Quadratic Mean Field Games with a Major Player: Nash certainty equivalence versus master equations\thanks{In a special issue of {\it Communications in Information and Systems}
in honor of  Professor Tyrone Duncan on the occasion of his 80th birthday.
Submitted May 27, 2020; revised September 27, 2020. This version: June 13, 2021.
 }
}

\vspace{1cm}

\author{Minyi Huang\thanks{M. Huang is with the School of
Mathematics and Statistics, Carleton University, Ottawa, ON K1S 5B6,
Canada (mhuang@math.carleton.ca). This work was supported in part  by Natural Sciences and Engineering
Research Council (NSERC) of Canada under a Discovery Grant.
The author would like to thank Prof. P. E. Caines and Dr. D. Firoozi for helpful discussions.}
     }

\date{}

\maketitle

\begin{abstract}
Mean field games with a major player were introduced in \cite{H10} within a linear-quadratic (LQ) modeling framework. Due to the rich structure of  major-minor player models, the past ten years have seen  significant research efforts for different solution notions and  analytical techniques. For LQ models,
we address the relation between three solution frameworks: the Nash certainty equivalence (NCE) approach in \cite{H10},  master equations, and asymptotic solvability, which  have been  developed starting with different ideas. We establish their equivalence relationships.
\end{abstract}

\begin{center}
\begin{minipage}{15.3cm}
\end{minipage}
\end{center}


\section{Introduction}

Mean field game theory was first developed for a population of
comparably small but possibly heterogeneous players; see an overview in \cite{CHM17}.
A generalization of this theory is to consider
a major player interacting with a large number of minor players
as initially introduced in \cite{H10}. Historically, games with major and minor players have been  studied in the literature, but usually for static cooperative
decision models (see e.g. \cite{H73,MS78}).

When a major player interacts with a large number of minor players,
a crucial  feature  is that the
mean field generated by the minor players is a random process even if the
population size of the minor players tends to infinity.
This necessitates the characterization  of the dynamic property of the mean field when the players choose their strategies based on mean field approximations.
This analytical complication is similar to what happens in common noise models \cite{BFY13,C12,CD18}.

Following the major-minor player mean field game model introduced in  \cite{H10},  the same LQ framework has been greatly extended by different authors.    The reader is referred to
\cite{NH12} for non-uniform minor players, \cite{CK17,FC15} for partial information,
\cite{HJN19} for an application to optimal execution in finance, \cite{HWW16} for system dynamics via backward stochastic differential equations,
\cite{KP15} for random entrance of agents, \cite{MH20} for multi-scale analysis and  the notion of asymptotic solvability, and \cite{FJC18} for a convex analysis approach.
Meanwhile, the study of major-minor players in nonlinear models can be found in \cite{BCY16,BFY13,BLP14,CW17,CZ16,NC13}. Nourian and Caines \cite{NC13}
 treat the mean field as a random measure flow driven by the major player's Brownian motion and use stochastic Hamilton-Jacobi-Bellman (HJB) equations to solve the best responses. Bensoussan et al \cite{BCY16} use stochastic adjoint equations to handle a pair of optimal control problems with
a dominating player.
 Carmona and Zhu \cite{CZ16} apply a
stochastic  maximum principle under a conditional law of the minor player.
Sen and Caines \cite{SC16} consider partial information
and control with nonlinear filtering.
Lasry and Lions \cite{LL18} introduce master equations for a nonlinear major-minor player model.
  Cardaliaguet et al \cite{CCP18} analyze
a convergence problem for a major-minor player model with nonlinear cost integrands, and  show that the pair of master equations can be obtained as the limit of the HJB equations for the $N+1$ players as $N\to \infty$.

The analysis of leadership or Stackelberg equilibria can be found in \cite{WZ14} for a discrete time model, \cite{BCL17,MB18} for LQ Stackelberg games,
\cite{BCY15} for control delay,
 \cite{HY20} for state feedback and time-consistency, \cite{EMP19} for mean field principal-agent problems, \cite{FH18} for an application to optimal portfolio liquidation, and \cite{K17} for evolutionary inspection games under a major player's pressure.

For  LQ mean field games, the current literature has provided several routes:  The Nash certainty equivalence (NCE) approach,  master equations, and asymptotic solvability. These solution methods start with  different  ideas although they all look for a certain Nash equilibrium with state feedback information.

The NCE approach was initially developed for mean field games with all comparably small players. One starts by considering an infinite population  and fixing the mean field \cite{HCM07,HMC06}. Next a special stochastic optimal control problem for a representative agent is solved by finding its best response strategy  with respect to the given mean field and subsequently letting the overall population implementing such strategies regenerate the same mean field, leading to a fixed point problem. The set of strategies obtained  is an $\varepsilon$-Nash equilibrium for the  finite population.

When extending the NCE approach in \cite{HMC06,HCM07} to the major player model, the  method in \cite{H10} is to augment the state space by an extra state describing the mean field  generated by different types of minor players. This Markovianizes the decision problems
to be solved by the major player and a representative minor player.
 A key step is to assume a linear structure of the mean field  dynamics and further  find appropriate constraints on its parameters by imposing consistency conditions.

 The master equation approach assigns the player in question with its own value function and next determines the optimization rule of all players. The mean field evolution is naturally determined by system dynamics under the chosen  strategies.

The notion of asymptotic solvability attempts to understand the large population decision problem from a different point of view. It shares similarity to convergence problems in \cite{CDLL15,F17,L16}.
 One can formally start to write the dynamic programming equations for a sequence of Nash games with population sizes tending to infinity. If for all sufficiently large populations, the game  has a solution via the coupled Riccati equations
and moreover, the Riccati equation solutions maintain certain boundedness properties, we say the sequence of games has asymptotic solvability. A basic question is  how to characterize asymptotic solvability in terms of some low dimensional structure that captures all essential information of the model. This question has been answered for LQ mean field games   in \cite{HZ20} without a major player and in \cite{MH20} with a major player.

When taking the approaches of
Nash certainty equivalence or master equations or asymptotic solvability, one typically obtains seemingly different solution structures.
 We are interested in studying the relationship between these solution equations,
 and their uniqueness properties.
This will shed light into the intrinsic nature  of these solutions and will be especially valuable when
 many solution formulations via different information or interaction patterns
exist for major-minor player models
and when non-uniqueness results of mean field games  have been frequently seen in the literature \cite{BF17,BZ20,GNP18,HL19,HZ20,T18}.
The main contribution of the paper is to show that the NCE equation system is equivalent to the master equations when the latter are restricted to quadratic solutions, and that for homogeneous minor players, the NCE equation system has a solution if and only if asymptotic solvability holds.  In contrast, for LQ mean field games without a major player, the NCE equation system via consistent mean field approximations may have a solution but asymptotic solvability fails \cite{HZ20}.

The organization of the paper is as follows.
Section \ref{sec:mod} presents a summary of the NCE approach adopted in \cite{H10}. Section \ref{sec:me} introduces the master equations for the LQ mean field games with a major player and $K$ types of minor players and compares with the NCE equation system. Section \ref{sec:as} analyzes the special case of homogeneous minor players; it first reviews the asymptotic solvability
problem in \cite{MH20} and  next   shows an equivalence relationship between the solution of  the NCE approach and asymptotic solvability. Section \ref{sec:con} concludes the paper.

{\it Notation:}
We use  $(\Omega, {\mathcal F}, \{{\mathcal F}_t\}_{t\ge 0}, P)$ to denote an underlying filtered probability space.
Let $S^n$ be  the set of $n\times n$ real and symmetric matrices,  $S_+^n$ its subset of
positive semi-definite matrices, and $I_k$ the $k\times k$ identity matrix.
 Let ${\mathcal P}_2(\mathbb{R}^n)$ be the set of probability measures $\mu$ on $\mathbb{R}^n$  that have finite second moment. Denote $\langle y\rangle_\mu=\int_{\mathbb{R}^n} y \mu(dy)$ for the probability measure $\mu$.
Given a symmetric matrix $M\geq 0$, the quadratic form $z^TMz$ may
be denoted as $|z|_M^2$.
  For  $ Z=(z_{jk}) \in \mathbb{R}^{l\times m}$, denote the $l_1$-norm $\| Z\|_{l_1}=\sum_{j,k}|z_{jk}|$.
 For matrices $A=(a_{ij})\in \mathbb{R}^{l_1\times l_2}$, $\hat A\in \mathbb{R}^{l_3 \times l_4}$, the Kronecker product
 $A\otimes \hat A= (a_{ij} \hat A)_{1\le i\le l_1,1\le j\le l_2}\in
 \mathbb{R}^{(l_1l_3)\times (l_2 l_4)}$.

\section{The LQ Mean Field Game with a Major Player}
\label{sec:mod}

This section summarizes the methodology developed in \cite{H10} which considered an infinite horizon. The model below takes a finite horizon  for convenience of later comparison with  master equations and asymptotic solvability by analyzing ordinary differential equations (ODEs). The approach in \cite{H10} is applied  in a straightforward manner.

We consider the LQ game with a major player ${\mathcal A}_0$
and $N$
minor players ${\mathcal A}_i$,  $1\le i\le N$.
The states of ${\mathcal A}_0$ and ${\mathcal A}_i$ are, respectively,
denoted by $X_0(t)$ and $X_i(t)$,  $1\le i\le N$, which satisfy the linear stochastic
differential equations (SDEs):
\begin{align}
dX_0(t)=\  &\big(A_0X_0(t)+B_0u_0(t)+F_0X^{(N)}(t)\big)dt
 +D_0dW_0(t),\label{maj}\\
dX_i(t)=\  &\big(A_{\theta_i}X_i(t)+Bu_i(t)+FX^{(N)}(t)+GX_0(t)\big)dt
+DdW_i(t),\label{min} \\
& 1\le i\le N,\quad t\ge 0,  \nonumber
\end{align}
where we have state  $X_j\in \mathbb{R}^n$, control
$u_j\in\mathbb{R}^{n_1}$, and
 $ X^{(N)}=\frac{1}{N}\sum_{k=1}^N X_k(t)$.
 The initial states $\{X_j(0), 0\le j\le N\}$ are independent with $EX_j(0)=x_j(0)$ and finite second moment. The $N+1$ standard  $n_2$-dimensional Brownian motions $\{W_j, 0\le j\le N\}$ are independent and also independent of the  initial states.
 The subscript $\theta_i$  is a dynamic parameter to model a heterogenous population of  minor players.
We assume $\theta_i$ takes values from the finite set $\Theta=\{1,\ldots, K\}$ modeling $K$ types of minor players. If $\theta_i=k$, ${\mathcal A}_i$ is called a $k$-type minor player.
 The  constant matrices
$A_0$, $B_0$, $F_0$, $D_0$, $A_k$,  $B$, $F$, $G$, $D$   have compatible
dimensions. Denote $u=(u_0, \cdots, u_N)$.
The costs of the players   are given by
\begin{align}
J_0(u) =\  & E\int_0^T e^{-\rho t}\left[  |X_0(t)-\Gamma_0 X^{(N)}(t)-\eta_0|_{Q_0}^2+|u_0(t)|_{R_0}^2\right]dt \nonumber\\
&+e^{-\rho T }E|X_0(T)-\Gamma_{0f} X^{(N)}(T)-\eta_{0f}|_{Q_{0f}}^2,\label{costJ0} \\
J_i(u) =\ &  E\int_0^Te^{-\rho t} \left[  |X_i(t)-\Gamma_1 X_0(t)-\Gamma_2 X^{(N)}(t)-\eta|_{Q}^2+|u_i(t)|^2_R \right]dt\nonumber\\
& +e^{-\rho T}E|X_i(T)-\Gamma_{1f} X_0(T)-\Gamma_{2f} X^{(N)}(T)-\eta_f|_{Q_f}^2,
\qquad  1\le i\le N, \label{costJi}
\end{align}
where $\rho\ge 0$ is a discount factor.
The  constant matrices (or vectors)
 $Q_0$, $\Gamma_0$, $\eta_0$, $R_0$, $Q_{0f}$, $\Gamma_{0f}$, $\eta_{0f}$, $Q$, $\Gamma_1$, $\Gamma_2$, $\eta$, $R$, $Q_{f}$, $\Gamma_{1f}$, $\Gamma_{2f}$, $\eta_f$ above
have compatible dimensions, and $Q_{0}\geq 0$, $Q_{0f}\geq 0$, $Q\geq 0$, $Q_f\geq 0$, $R_0>0$, $R>0$.
 Our analysis can be easily extended to
the case of time-dependent parameters.  We only take $A_{(\cdot)}$ in \eqref{min}  to be dependent on $\theta_i$ for the purpose of notational simplicity. When other parameters for ${\mathcal A}_i$ also depend on $\theta_i$, the analysis is similar.

For  given $N$, define
${\mathcal I}_k=\{i|\theta_i=k, 1\leq i\leq N\}$,
$N_k=|{\mathcal I}_k|$,
 where $|{\mathcal I}_k|$ is the cardinality of  ${\mathcal I}_k$, $1\leq k\leq K$.
Let  $ \pi_k^{(N)}=N_k/N$. Then $\pi^{(N)}=(\pi_1^{(N)}, \cdots, \pi_K^{(N)})$ is  the empirical distribution of $\theta_1,\cdots , \theta_N$.
We make the following assumptions.

({\bf A1}) There exists a  probability vector $\pi$ such that
 $\lim_{N\rightarrow \infty}\pi^{(N)}= \pi$,
where $\pi=(\pi_1,\cdots, \pi_K)$ and  $\min_{1\leq k\leq K}\pi_k>0$.

({\bf A2}) The initial states  $X_j(0)$, $0\leq j\leq N$, are independent, $EX_i(0)=\alpha_0$  for all $i\geq 1$, and there exists a fixed constant $c_0$  such that $\sup_{j\geq 0} E|X_j(0)|^2\leq c_0$.

\subsection{The limiting two-player  model}
\label{sec:lim}

Below we overview the steps  in \cite{H10}.
The key idea is to introduce  a new process $\bar Z(t)=[\bar Z_1^T(t), \cdots, \bar Z_K^T(t)]^T$ as a state component in an augmented state space,  where $\bar Z_k$ specifies the mean field generated by all $k$-type minor players.
Given the major player's state $\bar X_0$,  the process $\bar Z$ is described by the following equation
\begin{align}
d\bar Z(t)=[\overline A (t)\bar Z(t)  +\overline G(t) \bar X_0(t)  + \overline m (t)] dt,  \label{ztg}
\end{align}
where $\bar Z_k(0)=\alpha_0$, and  $\overline A(t)\in \mathbb{R}^{nK\times nK}$, $\overline G(t)\in {\mathbb R}^{nK\times n}$, and
$\overline m(t)$ are  continuous matrix or vector valued functions on $[0,T]$. The  initial condition for \eqref{ztg} is due to the  initial mean for the minor players as specified in ({\bf A2}). The triple $(\overline A, \overline G, \overline m)$
is not known in advance and  needs to be determined as part of the solution.

After replacing  $X^{(N)}$   in \re{maj}--\re{costJi} by $\sum_{k=1}^K \pi_k \bar Z_k$, we introduce two limiting optimal control problems for the major player and a representative minor player.
Let $\bar{\mathcal A}_0$ and $\bar{\mathcal A}_i$  stand for the two players.

Problem (P1):
The major player $\bar{\mathcal A}_0$ with state $\bar X_0$  has dynamics and cost
\begin{align}
& d\bar X_0 = \Big[A_0 \bar X_0+B_0 u_0+ F_0 \sum_{k=1}^K \pi_k \bar Z_k \Big]dt   +D_0 dW_0,  \nonumber \\ 
&d\bar Z=(\overline A \bar Z  +\overline G \bar X_0  + \overline m  )dt,
\qquad\qquad t\geq 0, \nonumber \\
& \bar J_0(u_0(\cdot))
=   E\int_0^T e^{-\rho t}\Big\{ \Big|\bar X_0- \Gamma_0
\sum_{k=1}^K \pi_k\bar Z_k-\eta_0  \Big|_{Q_0}^2
 + u_0^T R_0 u_0 \Big\}dt, \nonumber \\
 &\qquad \qquad \qquad+e^{-\rho T }E\Big|\bar X_0(T)-\Gamma_{0f}\sum_{k=1}^K \pi_k \bar Z_k (T)-\eta_{0f}\Big|_{Q_{0f}}^2, \nonumber
\end{align}
where  $\bar X_0(0)=X_0(0)$ and $\bar Z_k(0)= \alpha_0$.

Problem (P2):
 The minor player $\bar{\mathcal A}_i$ with state $\bar X_i$ has dynamics and cost
{\allowdisplaybreaks
\begin{align}
&d\bar X_i=  \Big[A_{\theta_i} \bar X_i +Bu_i + F\sum_{k=1}^K \pi_k \bar Z_k +  G  \bar X_0 \Big] dt
        +  DdW_i,   \label{minl} \\
 & d\bar X_0 = \Big[A_0 \bar X_0+B_0 \hat u_0+ F_0 \sum_{k=1}^K \pi_k \bar Z_k \Big]dt   +D_0 dW_0, \notag  \\
&d\bar Z=(\overline A \bar Z  +\overline G \bar X_0  + \overline m )
dt, \qquad\qquad t\geq 0,\nonumber \\
   \bar J_i(u_i(\cdot), \hat u_0(\cdot))=&
 E\int_0^T e^{-\rho t}\Big\{\Big|\bar X_i-\Gamma_1 \bar X_0- \Gamma_2\sum_{k=1}^K \pi_k\bar Z_k  -\eta\Big|_Q^2
  + u_i^T R u_i \Big\}dt, \nonumber\\
  & +e^{-\rho T}E\Big|\bar X_i(T)-\Gamma_{1f} \bar X_0(T)-\Gamma_{2f} \sum_{k=1}^K \pi_k \bar Z_k(T)-\eta_f\Big|_{Q_f}^2, \nonumber
\end{align}}
where $\bar X_i(0)=X_i(0) $,  $\bar X_0(0)=X_0(0) $, $\bar Z_k(0)=\alpha_0 $, and  $\hat u_0$ is the optimal control law solved from (P1).

To distinguish from the original model with $N+1$ players, we use the new state variables $\bar X_0$ and $\bar X_i$.  But we still reuse the same set of  variables  $u_0$, $u_i$, $W_0$ and $W_i$  in this infinite population limit model without the  risk of confusion.

\subsection{Nash certainty equivalence solution }
\label{sec:sub:nce}

Step 1: Solution of Problem (P1)

We start by assuming that $(\overline A,\overline G,\overline m)$ has been known.
Define
\begin{align*}
&F_0^\pi= \pi \otimes F_0, \quad
\Gamma_0^\pi = \pi \otimes \Gamma_0,\quad \Gamma_{0f}^\pi = \pi \otimes \Gamma_{0f},
\quad    \\
&\mathbb{A}_0(t)=\left[
\begin{array}{cc}
A_0 & F_0^\pi\\
\overline G & \overline A
\end{array}\right],
\quad
\mathbb{B}_0=\left[
\begin{array}{c}
B_0\\
0_{nK\times n_1}
\end{array}\right],
\quad
  \mathbb{M}_0(t)=\left[
\begin{array}{c}
0_{n\times 1}\\
\overline m
\end{array}\right],\nonumber \\
&{\mathbb Q}_0^\pi=[I, -\Gamma_0^\pi]^T Q_0 [I, -\Gamma_0^\pi],\qquad \bar \eta_0^\pi= [I,-\Gamma_0^\pi]^T Q_0  \eta_0,
\nonumber \\
& {\mathbb Q}_{0f}^\pi=[I, -\Gamma_{0f}^\pi]^T Q_{0f} [I, -\Gamma_{0f}^\pi],\qquad \bar \eta_{0f}^\pi= [I,-\Gamma_{0f}^\pi]^T Q_{0f}  \eta_{0f}.
\end{align*}

We introduce the ODE system
\begin{align}
&\rho P_0=\frac{dP_0}{dt}+ P_0\mathbb{A}_0+\mathbb{A}^T_0 P_0 -P_0 \mathbb{B}_0 R_0^{-1}\mathbb{B}_0^T P_0 +{\mathbb Q}_0^\pi,\label{riccatip0}\\
&\rho s_0= \frac{ds_0}{dt} +(\mathbb{A}_0^T-P_0\mathbb{B}_0 R_0^{-1} \mathbb{B}_0^T) s_0+P_0 \mathbb{M}_0-\bar\eta_0^\pi, \label{odes0}
\end{align}
where  $ P_0(T)= {\mathbb  Q}_{0f}^\pi $ and  $s_0(T)=- \bar \eta^\pi_{0f}$.
The ODE system has a unique solution.
 The optimal control law for $\bar{\mathcal A}_0$ is
\begin{align}
\hat u_0=-R_0^{-1} \mathbb{B}_0^T \left[P_0 (\bar X_0^T, \bar Z^T)^T+s_0\right]. \nonumber
\end{align}

Step 2: Solution of Problem (P2)

We suppose $\bar {\mathcal A}_i$ has its dynamic parameter $\theta_i=\kappa$
so that $A_{\theta_i}=A_\kappa$. Denote $F^\pi=\pi\otimes F$.
After taking the feedback  control law
$\hat u_0$ for $\bar {\mathcal A}_0$, we have

\begin{align}
d\left[
\begin{array}{c}
\bar X_i \\
\bar X_0\\
\bar Z
\end{array}
\right]&=    \left[
\begin{array}{cc}
A_\kappa & [G \ \ \ F^\pi]\\
0_{n(K+1)\times n} & \mathbb{A}_0-\mathbb{B}_0 R_0^{-1} \mathbb{B}_0^T P_0
\end{array}
\right]\left[
\begin{array}{c}
\bar X_i \\
\bar X_0\\
\bar Z
\end{array}
\right]dt\nonumber\\
&\quad+\left[\!
\begin{array}{c}
B\\
0_{n(K+1)\times n_1}
\end{array}
\!\right] u_i dt
  +\left[\!
\begin{array}{c}
0_{n\times 1}\\
\mathbb{M}_0 -\mathbb{B}_0R_0^{-1} \mathbb{B}_0^T s_0
\end{array}
\!\right]dt+
\left[\!
\begin{array}{c}
DdW_i \\
D_0 dW_0\\
0_{nK\times 1}
\end{array}
\!\right],  \nonumber
\end{align}
where $\bar X_i(0)=X_i(0)$, $\bar X_0(0)=X_0(0)$, $\bar Z(0)=\alpha_0$ and   $P_0$ is solved from \re{riccatip0}.
Define
\begin{align*}
&{\mathbb{A}}_{\kappa}(t)= \left[
\begin{array}{cc}
A_\kappa & [G \ \ \ F^\pi]\\
0_{n(K+1)\times n} & \mathbb{A}_0-\mathbb{B}_0 R_0^{-1} \mathbb{B}_0^T P_0
\end{array}
\right], \quad \mathbb{B}=\left[
\begin{array}{c}
B\\
0_{n(K+1)\times n_1}
\end{array}
\right] ,
\nonumber \\
& {\mathbb{M}}(t)= \left[
\begin{array}{c}
0_{n\times 1}\\
\mathbb{M}_0-\mathbb{B}_0R_0^{-1} \mathbb{B}_0^T s_0
\end{array}
\right],\quad \Gamma_2^\pi=\pi\otimes \Gamma_2,\quad \Gamma_{2f}^\pi=\pi\otimes \Gamma_{2f}, \\ 
 &  {\mathbb  Q}^\pi= [I, -\Gamma_1, -\Gamma_2^\pi]^T Q[I, -\Gamma_1, -\Gamma_2^\pi], \quad \bar\eta^\pi= [I, -\Gamma_1, -\Gamma_2^\pi]^TQ  \eta, \nonumber\\
& {\mathbb Q}^\pi_f= [I, -\Gamma_{1f}, -\Gamma_{2f}^\pi]^T Q_f[I, -\Gamma_{1f}, -\Gamma_{2f}^\pi], \quad \bar\eta_f^\pi= [I, -\Gamma_{1f}, -\Gamma_{2f}^\pi]^TQ_f  \eta_f.
\end{align*}

We introduce
\begin{align}
&\rho P_{\kappa} =\frac{d P_\kappa}{dt} +P_{\kappa}{\mathbb{A}}_{\kappa} +{\mathbb{A}}_{\kappa}^T
P_{\kappa} - P_{\kappa} {\mathbb{B}} R^{-1}{\mathbb{B}}^T P_{\kappa} +{\mathbb Q}^\pi, \label{Pk5}\\
&\rho  s_\kappa= \frac{d s_\kappa}{dt} +\left(\mathbb{A}_{\kappa}^T-P_{\kappa} {\mathbb{B}} R^{-1} {\mathbb{B}}^T\right)  s_\kappa+P_{\kappa}{\mathbb{M}}-\bar\eta^\pi, \label{sk5}
\end{align}
where $P_\kappa(T)={\mathbb Q}^\pi_f $ and  $s_\kappa(T)=  -\bar\eta_f^\pi$.
The optimal control law for $\bar{\mathcal A}_i$ is  given by
\begin{align}
\hat u_i= -R^{-1} {\mathbb{B}}^T \left[P_{\kappa} (\bar X_i^T, \bar X_0^T, \bar Z^T)^T+  s_\kappa\right]. \label{mincontrol}
\end{align}

Finally, substituting  \re{mincontrol} into  \re{minl} gives
\begin{align}
d\bar X_i =\ & (A_\kappa \bar X_i + G \bar X_0 + F^\pi \bar Z) dt -BR^{-1} {\mathbb{B}}^T P_{\kappa} (\bar X_i^T, \bar X_0^T, \bar Z^T)^T dt\nonumber\\
&-   BR^{-1} {\mathbb{B}}^T s_\kappa dt +D dW_i, \label{xico}
\end{align}
where $\bar X_i(0)=X_i(0)$.

{Step 3: The consistency  condition}

For the matrices $P_\kappa$, $\kappa=1,\cdots, K$,  we introduce the
partition
\begin{align}
P_\kappa=(P_{\kappa,jl})_{1\le j,l\le 3},
 \label{Ppart}
\end{align}
where $P_{\kappa,11}$, $P_{\kappa,22}\in \mathbb{R}^{n\times n}$,  and $ P_{\kappa,33}\in \mathbb{R}^{nK\times nK}$.
The matrix  functions $\overline A(t)$, $\overline G(t)$ and
vector  function $\overline m(t)$ are represented  in the form
\begin{align}
\overline A(t)=\left[
\begin{array}{c}
\overline A_1\\
\vdots\\
\overline A_K
\end{array}\right], \qquad  \overline G(t)=\left[
\begin{array}{c}
\overline G_1\\
\vdots\\
\overline G_K
\end{array}\right],  \qquad  \overline m(t)=\left[
\begin{array}{c}
\overline m_1\\
\vdots\\
\overline m_K
\end{array}\right], \label{xiAG}
\end{align}
where  $\overline A_k(t)\in \mathbb{R}^{n\times nK}$, $\overline G_k(t)\in \mathbb{R}^{n\times n}$ and  $\overline m_k(t)\in \mathbb{R}^{n}$ for $1\leq k\leq K$.
Denote 
\begin{align}
{\bf e}_k=[0_{n\times n}, \cdots, 0_{n\times n}, I_{n}, 0_{n\times n},\cdots, 0_{n\times n}]\in \mathbb{R}^{n\times nK}, \label{eI}
 \end{align}
 where the identity matrix $I_{n} $ is at the $k$th block, $1\leq k\leq K$.
Now we consider the  average state $(1/N_\kappa) \sum_{i\in {\mathcal I}_\kappa} X_i$  of  $N_\kappa$ $\kappa$-type minor players with  closed-loop dynamics of the form \re{xico}. When $N\rightarrow \infty$ so that $N_\kappa \rightarrow  \infty$, the limit of the state average is required to regenerate  $\bar Z_\kappa$ (which has been intended for the approximation of $(1/N_\kappa) \sum_{i\in {\mathcal I}_\kappa} X_i$) such that
\begin{align}
d\bar Z_\kappa =\
           &  \left\{\left[A_\kappa- BR^{-1} B^T P_{\kappa, 11}\right]{\bf e}_\kappa
          +F^\pi- BR^{-1}B^TP_{\kappa, 13}\right\}\bar Zdt\nonumber\\
& + (G-BR^{-1}B^T P_{\kappa, 12})\bar X_0dt
- BR^{-1} \mathbb{B}^T s_\kappa dt, \label{zkco}
\end{align}
where $\bar Z=[\bar Z_1^T,\cdots, \bar Z_K^T]^T$ and $\bar Z_\kappa(0)=\alpha_0$.
Now, under the  NCE methodology, the resulting equation system \re{zkco}  is required to coincide with \re{ztg} which had been presumed in the first place.  We call this requirement  the {\it consistency condition}.

\subsection{The Nash certainty equivalence equation system}

Based on Steps 1, 2 and 3, we
 introduce the first differential-algebraic  system of equations (DAEs):
\begin{align}
\left\{
\begin{array}{l}
\rho P_0=\dot{P}_0+ P_0\mathbb{A}_0+\mathbb{A}_0^T P_0 -P_0 \mathbb{B}_0 R_0^{-1}\mathbb{B}_0^T P_0 +{\mathbb Q}_0^\pi,\\
\rho P_{\kappa} =\dot{P}_\kappa+ P_{\kappa} {\mathbb{A}}_{\kappa} +{\mathbb{A}}_{\kappa}^T
P_{\kappa} - P_{\kappa} {\mathbb{B}} R^{-1}{\mathbb{B}}^T P_{\kappa} +{\mathbb Q}^\pi,\quad
1\le \kappa\le K,\\
  \overline A_\kappa= (A_\kappa- BR^{-1} B^T P_{\kappa, 11}) {\bf e}_\kappa + F^\pi- BR^{-1} B^TP_{\kappa, 13} , \qquad \forall \kappa, \\
 \overline G_\kappa =G-BR^{-1} B^T P_{\kappa, 12} , \qquad \forall\kappa,
\end{array}\right. \label{conricK}
\end{align}
where $ P_0(T)= {\mathbb  Q}_{0f}^\pi $ and $P_\kappa(T)={\mathbb Q}^\pi_f$,
and
the  second differential-algebraic system of equations (DAEs):
\begin{align}
\left\{
\begin{array}{l}
\rho s_0= \dot{s}_0 +(\mathbb{A}_0^T-P_0\mathbb{B}_0 R_0^{-1} \mathbb{B}_0^T) s_0+P_0 \mathbb{M}_0-\bar\eta_0^\pi,
\\
\rho  s_\kappa= \dot{s}_\kappa +(\mathbb{A}_{\kappa}^T-P_{\kappa} {\mathbb{B}} R^{-1} {\mathbb{B}}^T)  s_\kappa+P_{\kappa} {\mathbb{M}}-\bar\eta^\pi, \quad  1\le \kappa\le K,\\
\overline m_\kappa =-   BR^{-1} {\mathbb{B}}^T s_\kappa, \qquad \forall \kappa,
\end{array}\right.\label{conodeK}
\end{align}
where  $s_0(T)=- \bar \eta^\pi_{0f}$ and  $s_\kappa(T)=  -\bar\eta_f^\pi$.
 Note that $\overline m$ has been used in defining ${\mathbb M}_0$ and ${\mathbb{M}}$.
 The equality constraints for $\overline A_\kappa$, $\overline
 G_\kappa$ in \eqref{conricK} and $\overline m_\kappa$ in \eqref{conodeK} result from the consistency condition specified in Step 3.
The combined differential-algebraic  system of equations \re{conricK}--\re{conodeK} will be called the NCE  equation system.
The matrices ${\mathbb A}_0$ and ${\mathbb A}_\kappa$ now depend on the solution of the equation system. The solution of \eqref{conricK}, if it exists, can be solved without involving \eqref{conodeK}.

Denote
 \begin{align*}
 {\mathcal X}_P =\ & C([0,T]; \mathbb{R}^{n(K+1)\times n(K+1)})\times ( C([0,T]; \mathbb{R}^{n(K+2)\times n(K+2)}))^K\\
&\times C([0,T]; \mathbb{R}^{nK\times nK})\times C([0,T]; \mathbb{R}^{nK\times n}) , \\
{\mathcal X}_s=\ &  C([0,T]; \mathbb{R}^{n(K+1)})\times (C([0,T]; \mathbb{R}^{n(K+2)}))^K\times C([0,T]; \mathbb{R}^{nK}).
\end{align*}

\begin{definition}\label{def:sol}
A set of functions
$$(P_0, P_1, \cdots, P_K,\overline A, \overline G)\in {\mathcal X}_P,  \quad ( s_0, s_1,\cdots, s_K,\overline  m) \in  {\mathcal X}_s
$$
 satisfying \eqref{conricK}--\eqref{conodeK} on $[0,T]$ is called a solution of  the NCE equation system.
\end{definition}

\begin{lemma}\label{lemma:NCE}
We have the following assertions:

{\emph {i)}} If $(P_0, P_1,\cdots, P_K, \overline A, \overline G)\in {\mathcal X}_P$ is a solution of \eqref{conricK}, then it is unique and
 $P_0\in C([0,T]; S^{n(K+1)}_+) $, $P_\kappa\in C([0,T]; S^{n(K+1)}_+) $ for each $1\le \kappa\le K$.

{\emph {ii)}} The NCE equation system has a solution if and only if
\eqref{conricK} has a solution.

{\emph{iii)}} If the NCE equation system has a solution in ${\mathcal X}_P\times {\mathcal X}_s$, it is unique.

\end{lemma}

\proof i) By eliminating  $(\overline A, \overline G)$, we write $\mathbb{A}_0$ as a function of $(P_1, \cdots, P_K)$, and $\mathbb{A}_\kappa$ as  a function of
$(P_0, P_1, \cdots, P_K)$, and accordingly  denote
${\mathbb A}_0({P_1}, \cdots, P_K)$,  ${\mathbb A}_\kappa(P_0, P_1, \cdots, P_K)$. Next we write the ODEs of $(P_0, P_1, \cdots, P_K)$ where the vector field is a continuous  function of $(P_0, P_1, \cdots, P_K)$ with local Lipschitz continuity. Hence the
 solution $(P_0, P_1, \cdots, P_K)$ is unique, which further uniquely determines $(\overline A, \overline G)$. After solving $P_0(t)$ and $P_1(t), \cdots, P_K(t)$, we denote
\begin{align}
 \hat {\mathbb A}_0(t)= {\mathbb A}_0 (P_1(t), \cdots, P_K(t)), \quad \hat {\mathbb A}_\kappa(t)={\mathbb A}_\kappa(P_0(t), P_1(t), \cdots, P_K(t)).\notag
\end{align}
Consider the standard Riccati ODE
\begin{align}
\rho \hat P_0= \dot {\hat P}_0+ \hat P_0
\hat {\mathbb A}_0+\hat {\mathbb A}_0^T \hat P_0 -\hat P_0 \mathbb{B}_0 R_0^{-1}\mathbb{B}_0^T \hat P_0 +{\mathbb Q}_0^\pi, \qquad \hat P_0(T)= {\mathbb Q}_{0f}^\pi. \label{hRic}
\end{align}
We obtain a unique solution $\hat P_0\in C([0,T]; S^{n(K+1)}_+)$ \cite{S98}.
Since $P_0$ also satisfies \eqref{hRic}, we necessarily have
$\hat P_0=P_0$. It follows that $P_0\in C([0,T]; S^{n(K+1)}_+)$.
Similarly we obtain $P_\kappa\in C([0,T]; S^{n(K+2)}_+)$ for $1\le \kappa \le K$.

ii) Necessity is trivial. We show sufficiency.
After solving $(P_0, P_1, \cdots, P_K)$, we write ${\mathbb M}_0(s_1, \cdots, s_K)$ as a linear function of $s_1, \cdots, s_K$, and ${\mathbb M}(s_0, s_1, \cdots, s_K)$ as a linear function of $(s_0, s_1, \cdots, s_K)$. Then $(s_0, s_1, \cdots, s_K)$ is determined by a linear ODE system with time dependent coefficients and can be uniquely solved. We further solve $\overline m$.

iii) Suppose
$(P_0, P_1,\cdots, P_K \overline A, \overline G, s_0, s_1,\cdots, s_K,\overline  m)$ is a solution of the NCE equation system.  Let
$(P_0', P_1',\cdots, P_K', \overline A', \overline G', s_0', s_1',\cdots, s_K',\overline  m')$
be another solution. By part i), we have
$$(P_0, P_1,\cdots, P_K, \overline A, \overline G)= (P_0', P_1', \cdots, P_K',\overline A', \overline G').$$
 Consequently, $(s_0, s_1,\dots, s_K,\overline  m)  =( s_0', s_1',\cdots, s_K',\overline  m')$.
Uniqueness follows. \endproof

\begin{remark}
 When $T$ is replaced by $\infty$, the system specified by \re{maj}--\re{costJi} reduces to the model in \emph{\cite{H10}}, and
\eqref{conricK} is modified by replacing the two ODEs by two algebraic equations. For the infinite horizon case, \eqref{conodeK} can  be retained as  a differential-algebraic system of equations with no terminal condition, for which one can impose appropriate growth conditions on the solutions $(s_0, s_\kappa)$. The consideration of time-varying $s_0$, $s_\kappa$ means searching  in a larger space than constant solutions.
\end{remark}

\begin{remark}
 A class of LQ stochastic optimal control problems is solved in \emph{\cite{FJC18}} based on convex analysis and the G\^{a}teaux derivative of the cost; this approach is applied to solve the limiting optimal control problems of the LQ mean field game with a major player for both the finite and infinite horizon cases. This recovers both the NCE equation system obtained in \emph{\cite{H10}} and \eqref{conricK}--\eqref{conodeK}.
\end{remark}

\section{The Master Equations and Their Relation to the NCE Approach}
\label{sec:me}

Let $X_0^\dag(t)$ and $Z^\dag_\kappa(t)$, $1\le \kappa\le K$, be the states of the  major player and a minor player of type $\kappa$, respectively. Let $\{\mu^\dag_\kappa(t), 0\le t\le T\}$ denote the (random) measure flow as the mean field generated by all minor players of type $\kappa$.
Due to the correlation of the $\kappa$-type minor players' states, the
limit $\mu^\dag_\kappa(t)$  of their empirical distribution is still random except for trivial cases.

For $\mu= (\mu_1, \cdots, \mu_K) $, where
each $\mu_\kappa\in {\mathcal P}_2({\mathbb R}^n)$,  define the mean function
\begin{align}
\bar z_\kappa^\mu= \bar z_\kappa=\langle y \rangle_{\mu_\kappa}\in \mathbb{R}^n, \quad \bar z^\mu=[\bar z^T, \cdots, \bar z_K^T   ]^{T} .\label{zmu}
\end{align}
More generally, for the measure flow  $\{\mu^\dag(t)=(\mu_1^\dag(t), \cdots,\mu_K^\dag(t) ), 0\le t\le T\}$, we
denote $\bar z_\kappa^{\mu^\dag(t)}= \bar z_\kappa(t)=\langle y\rangle_{\mu^\dag_{\kappa}(t)}$ and $\bar z^{\mu^\dag}(t)=\bar z^{\mu^\dag(t)}=[\bar z_1^T(t),\cdots, \bar z_K^T(t)]^T$. For this section, $\mu$ and $\mu^\dag$ always contain $K$ components,
and this should be clear from the context.

We introduce the following model with $K+1$ players
with dynamics
\begin{align*}
dX_0^\dag& = (A_0X_0^\dag +B_0 u_0 +F_0^\pi \bar z^{\mu^\dag} ) dt+ D_0 dW_0, \qquad 0\le t\le T,\\
dZ_\kappa^\dag&=( A_{\kappa}Z_\kappa^\dag + B u_{\kappa}+GX_0^\dag +F^\pi \bar z^{\mu^\dag})dt +DdW_{\kappa}, \quad 1\le \kappa \le K,
\end{align*}
and costs
\begin{align*}
&  J_0(u_0, \cdots, u_K)
=   E\int_0^T e^{-\rho t}\Big\{ \Big| X_0^\dag- \Gamma_0^\pi \bar z^{\mu^\dag}-\eta_0  \Big|_{Q_0}^2
 + u_0^T R_0 u_0 \Big\}dt, \\ 
 &\qquad \qquad \qquad \qquad +e^{-\rho T }E\Big|X_0^\dag(T)-\Gamma_{0f}^\pi \bar z^{\mu^\dag}  (T)-\eta_{0f}\Big|_{Q_{0f}}^2,\\
&  J_\kappa(u_0,\cdots,  u_K)=
 E\int_0^T e^{-\rho t}\Big\{\Big| Z_\kappa^\dag-\Gamma_1 X_0^\dag- \Gamma_2^\pi \bar z^{\mu^\dag}  -\eta\Big|_Q^2
  + u_\kappa^T R u_\kappa \Big\}dt, \\
  &\qquad\qquad  +e^{-\rho T}E\Big|Z_\kappa^\dag(T)-\Gamma_{1f} X_0^\dag(T)-\Gamma_{2f}^\pi \bar z^{\mu^\dag}(T) -\eta_f\Big|_{Q_f}^2,\quad 1\le \kappa \le K.
\end{align*}
The parameters in the dynamics and costs are the same as in \eqref{maj}--\eqref{costJi}.
The major player takes $(X_0^\dag, \mu^\dag)$ as the state variable in its control problem, and the $\kappa$-type minor player takes $(Z_\kappa^\dag, X_0^\dag, \mu^\dag)$ as the state variable.

 Take the initial time $t$
 and initial state $( x_0, \mu)$ (resp., $( z_\kappa, x_0, \mu)$) for the major player (resp., the $\kappa$-type minor player).
Denote the value functions
$V_0(t, x_0, \mu)$, $V_\kappa(t,z_\kappa, x_0,  \mu)$,
where $t\in [0,T]$, $x_0, z_\kappa \in \mathbb{R}^n$ and $\mu= (\mu_1, \cdots, \mu_K)$, $\mu_\kappa \in {\mathcal P}_2(\mathbb{R}^n)$.

We have the Hamilton-Jacobi-Bellman (master) equations  
{\allowdisplaybreaks
\begin{align}
&-\partial_t V_0 +\rho V_0 \label{ME1h}  \\
 =\ & \partial_{x_0}^T V_0(A_0 x_0+B_0\hat u_0
+ F_0^\pi \bar z^{\mu} ) \nonumber  \\
  & +|x_0-\Gamma_0^\pi \bar z^{\mu}-\eta_0|^2_{Q_0}
  + \hat u_0^T R_0 \hat u_0 +(1/2) \mbox{Tr} ( \partial_{x_0x_0}V_0 D_0D_0^T )   \nonumber \\
  &+ \sum_{l=1}^K\int \partial_{y_l}^T (\partial_{\mu_l}V_0)(t, x_0, \mu; y_l) [A_{l}y_l +B\hat u_{l}(t,y_l,x_0,\mu) 
+Gx_0 +F^\pi\bar z^{\mu}]\mu_l(dy_l)\nonumber \\
& +\sum_{l=1}^K\int \frac{1}{2} {\rm Tr} [\partial_{y_l y_l} (\partial_{\mu_l}V_0)(t, x_0, \mu; y_l) DD^T]\mu_l(dy_l),\nonumber
\end{align}}
where $V_0(T, x_0, \mu)=|x_0- \Gamma_{0f}^\pi
\bar z^\mu -\eta_{0f}|_{Q_{0f}}^2  $,
and
\begin{align}
&-\partial_t V_\kappa +\rho V_\kappa  \label{ME2h}  \\
 =\ & \partial_{x_0}^T V_\kappa(A_0 x_0+B_0\hat u_0+F_0^\pi\bar z^{\mu}  ) +({1}/{2}) \mbox{Tr} ( \partial_{x_0x_0}V_\kappa D_0D_0^T )\nonumber\\
   &+ \partial_{z_\kappa}^T V_\kappa(A_{\kappa} z_k+B\hat u_{\kappa}+Gx_0+
   F^\pi\bar z^{\mu}   )
   \nonumber \\
  & +|z_\kappa-\Gamma_1 x_0-\Gamma_2^\pi \bar z^{\mu} -\eta|^2_{Q} + \hat u_\kappa^{T} R \hat u_{\kappa}   +({1}/{2})\mbox{Tr}(\partial_{z_\kappa z_\kappa}V_\kappa DD^T) \nonumber\\
  &+ \sum_{l=1}^K\int \partial_{y_l}^T (\partial_{\mu_l}V_\kappa)(t,z_k, x_0, \mu; y_l) [A_{l}y_l +B\hat u_{l}(t,y_l, x_0,\mu) 
+Gx_0 +F^\pi\bar z^{\mu}
]\mu_l(dy_l)\nonumber\\
& +\sum_{l=1}^K\int \frac{1}{2} {\rm Tr} [\partial_{y_ly_l} (\partial_{\mu_l}V_\kappa)(t,z_k, x_0, \mu; y_l) DD^T]\mu_l(dy_l),\nonumber
\end{align}
where $V_\kappa (T, z_\kappa, x_0,\mu)=  |z_\kappa - \Gamma_{1f} x_0-\Gamma_{2f}^\pi \bar z^\mu- \eta_f |_{Q_f}^2$.

We may regard \eqref{ME1h}--\eqref{ME2h} as $K+1$ dynamic programming equations which are formally derived by a local expansion of $V_0(t+\epsilon, X_0^\dag(t+\epsilon), \mu^\dag(t+\epsilon))$ and $V_\kappa(t+\epsilon, Z_\kappa^\dag(t+\epsilon), X_0^\dag(t+\epsilon), \mu^\dag(t+\epsilon))$ given $X_0^\dag (t)=x_0$, $Z^\dag_\kappa(t)=z_\kappa$, $\mu^\dag(t)=\mu$.
Each of the equilibrium strategies in  $(\hat u_0, \hat u_1, \cdots, \hat u_K)$ is selected as a best response to  maximize its own Hamiltonian. This amounts to finding
\begin{align}
&\hat u_0=\arg\max_{u_0} (- \partial_{x_0}^T V_0 B_0 u_0- u_0^T R_0 u_0) , \nonumber\\
&\hat u_\kappa=\arg\max _{u_\kappa} (- \partial_{z_\kappa}^T V_\kappa B u_{\kappa} -  u_\kappa^{T} R  u_{\kappa}). \label{ukarg}
\end{align}
The integral terms in \eqref{ME1h}--\eqref{ME2h} account for the variation of  the value function that is caused by the small perturbation of the mean field term. The choice of $\hat u_\kappa$ in \eqref{ukarg} does not directly use the integral terms since the control of the minor player in question has little impact on the mean field. The reader may consult \cite{CCP18,LL18} for master equations of nonlinear major player models, and \cite{CDLL15} for differentiation with respect to probability measures (if $K=1$, our notation $\partial_\mu V_0$ is equivalent to $ \frac{\delta V_0}{\delta \mu}$ in \cite{CDLL15}). Note that after differentiation of the value functions, a new variable $y_l$  arises so that we write
\begin{align}
(\partial_{\mu_l}V_0)(t, x_0, \mu; y_l), \qquad (\partial_{\mu_l}V_\kappa)(t,z_k, x_0, \mu; y_l). \label{dmuv0v}
\end{align}

We determine
$$B_0^T \partial_{x_0} V_0+2 R_0 \hat u_0=0,\qquad  B^T \partial_{z_\kappa}^T V_\kappa   +  2 R  \hat u_{\kappa}=0,$$
which gives
\begin{align}
&\hat u_0(t, x_0,\mu)= -({1}/{2}) R_0^{-1}B_0^T \partial_{x_0} V_0(t, x_0, \mu),  \label{u0hmu}\\
&\hat u_{\kappa}(t , z_\kappa , x_0,\mu)  = -({1}/{2}) R^{-1} B^T \partial_{z_\kappa} V_\kappa (t,z_\kappa, x_0,  \mu). \label{ukhmu}
\end{align}
Next, substituting \eqref{u0hmu}--\eqref{ukhmu} into \eqref{ME1h}--\eqref{ME2h} we obtain
\begin{align}
&-\partial_t V_0 +\rho V_0 \label{ME1} \\
 =\ & \partial_{x_0}^T V_0(A_0 x_0+F_0^\pi\bar z^\mu ) -({1}/{4})
  \partial_{x_0}^TV_0 B_0 R_0^{-1} B_0^T\partial_{x_0}V_0\notag\\
  &  +|x_0-
  \Gamma_0^\pi\bar z^\mu -\eta_0|^2_{Q_0} +({1}/{2}) \mbox{Tr} ( \partial_{x_0x_0}V_0 D_0D_0^T )  \notag \\
  &+ \sum_{l=1}^K\int \partial_{y_l}^T (\partial_{\mu_l}V_0)(t, x_0, \mu; y_l) (A_{l}y_l +Gx_0 +F^\pi\bar z^\mu)\mu_l(dy_l)\notag\\
&- \sum_{l=1}^K \int \frac{1}{2} \partial_{y_l}^T (\partial_{\mu_l}V_0)(t, x_0, \mu; y_l)BR^{-1} B^T \partial_{y_l} V_l (t, y_l, x_0, \mu)  \mu_l(dy_l)\nonumber\\
& +\sum_{l=1}^K\int \frac{1}{2} {\rm Tr} [\partial_{y_ly_l} (\partial_{\mu_l}V_0)(t, x_0, \mu; y_l) DD^T]\mu_l(dy_l), \notag
\end{align}
where $V_0(T, x_0, \mu)=|x_0- \Gamma_{0f}^\pi \bar z^\mu -\eta_{0f}|_{Q_{0f}}^2  $,
and {\allowdisplaybreaks
\begin{align}
&-\partial_t V_\kappa +\rho V_\kappa \label{ME2}   \\
 =\ & \partial_{x_0}^T V_\kappa(A_0 x_0  +F_0^\pi\bar z^\mu ) -({1}/{2})\partial_{x_0}^T V_\kappa B_0 R_0^{-1}B_0^T \partial_{x_0} V_0 \nonumber\\
&  +({1}/{2}) \mbox{Tr} ( \partial_{x_0x_0}V_\kappa D_0D_0^T ) \notag\\
   &+ \partial_{z_\kappa}^T V_\kappa(A_{\kappa} z_\kappa+Gx_0+F^\pi\bar z^\mu )-({1}/{4})\partial_{z_\kappa}^TV_\kappa B R^{-1} B^T\partial_{z_\kappa}V_\kappa \notag \\
  & +|z_\kappa-\Gamma_1 x_0-\Gamma_2^\pi\bar z^\mu -\eta|^2_{Q}+({1}/{2}) \mbox{Tr} ( \partial_{z_\kappa z_\kappa}V_\kappa DD^T )
\notag \\
  &+ \sum_{l=1}^K\int \partial_{y_l}^T (\partial_{\mu_l}V_\kappa)(t,z_\kappa, x_0, \mu; y_l) (A_{l}y_l +Gx_0 +F^\pi\bar z^\mu)\mu_l(dy_l) \notag \\
& - \sum_{l=1}^K \int \frac{1}{2} \partial_{y_{l}}^T (\partial_{\mu_l}
V_\kappa)(t,z_\kappa, x_0, \mu; y_{l})BR^{-1} B^T \partial_{y_l} V_l (t,y_l, x_0, \mu)  \mu_l(dy_l) \notag \\
& +\sum_{l=1}^K\int \frac{1}{2} {\rm Tr} [\partial_{y_ly_l}
(\partial_{\mu_l}V_\kappa)(t,z_\kappa, x_0, \mu; y_l) DD^T]\mu_l(dy_l), \nonumber
\end{align}}%
where $V_\kappa (T, z_\kappa, x_0,\mu)=  |z_\kappa - \Gamma_{1f} x_0-\Gamma_{2f}^\pi \bar z^\mu- \eta_f |_{Q_f}^2$.
We call \eqref{ME1} and \eqref{ME2} the master equations.

For the right hand side of \eqref{ME1}, we denote it as
$$
\chi_0\coloneqq\chi_{0,1}-\chi_{0,2} +\chi_{0,3}+\chi_{0,4} +\chi_{0,5}-\chi_{0,6} +\chi_{0,7},
$$
where each constituent term stands for the term in the master equation taking the same place. Similarly, the right hand side of \eqref{ME2} is written as
$$
\chi\coloneqq \chi_1-\chi_2+\chi_3 +\chi_4-\chi_5 +\chi_6 +\chi_7 +\chi_8-\chi_9+\chi_{10}.
$$

\subsection{Quadratic solutions}

Recalling \eqref{zmu},
for $\mu=(\mu_1, \cdots, \mu_K)$,  we now simply denote
$\bar z_\kappa =\langle y\rangle_{\mu_\kappa}\in \mathbb{R}^n $ and $\bar z=[\bar z_1^T, \cdots, \bar z_K^T]^T $.
Denote
$
\xi_0=
[
x_0^T, \bar z^T
]^T$ and
$\xi_\kappa=
[
z_\kappa^T,
x_0^T,
\bar z^T]^T.
$
We are interested in solutions of the following form:
\begin{align}
&V_0(t, x_0, \mu)= \xi_0^T {\bf P}_0^\dag(t) \xi_0+2 {\bf  s}_0^{\dag T}(t) \xi_0 +{\bf r}_0^\dag(t), \label{qs0} \\
& V_\kappa(t, z_\kappa,x_0, \mu)= \xi_\kappa^T {\bf P}_\kappa^\dag(t) \xi_\kappa+2 {\bf  s}_\kappa^{\dag T}(t) \xi_\kappa +{\bf r}_\kappa^\dag(t), \quad 1\le \kappa\le K, \label{qsk}
\end{align}
where ${\bf P}_0^\dag(t) $ and ${\bf P}_\kappa^\dag(t)$ are symmetric matrix functions of $t\in [0, T]$, and
the coefficient functions are differentiable on $[0,T]$.
Such a solution $(V_0, V_1, \cdots, V_K)$ is called a quadratic solution to the master equations.

Denote the partition
${\bf P}_0^\dag(t)= ({\bf P}^\dag_{0,jl})_{1\le j,l\le 2}
$
and $
{\bf P}_\kappa^\dag(t)=({\bf P}_{\kappa,jl}^\dag)_{1\le j,l\le 3}
$,
where ${\bf P}_{0,11}^\dag\in  \mathbb{R}^{n\times n}$, ${\bf P}_{0,22}^\dag\in \mathbb{R}^{nK\times nK}$, ${\bf P}_{\kappa,11}^\dag$, ${\bf P}_{\kappa,22}^\dag\in \mathbb{R}^{n\times n}$,   ${\bf  P}_{\kappa,33}^\dag\in \mathbb{R}^{nK\times nK}$.
Denote
\begin{align*}
{\bf s}_0^\dag(t)=
\begin{bmatrix}
{\bf s}_{0,1}^\dag\\
{\bf s}_{0,2}^\dag
\end{bmatrix}, \qquad
{\bf s}_\kappa^\dag(t)=
\begin{bmatrix}
{\bf s}_{\kappa,1}^\dag\\
{\bf s}_{\kappa,2}^\dag\\
{\bf s}_{\kappa,3}^\dag
\end{bmatrix},
\end{align*}
where ${\bf s}_{0,1}^\dag\in \mathbb{R}^{n}$,
${\bf s}_{0,2}^\dag\in \mathbb{R}^{nK}$, ${\bf s}_{\kappa,1}^\dag\in  \mathbb{R}^{n}$,
${\bf s}_{\kappa,2}^\dag\in \mathbb{R}^{n}$, and
${\bf s}_{\kappa,3}^\dag\in  \mathbb{R}^{nK}$.

In order to calculate the integral terms in the master equations to analyze quadratic solutions,
we introduce some notation. For $1\le l\le K$,
denote
\begin{align}
&\overline {\bf A}_l^\dag(t)= (A_l - BR^{-1} B^T{\bf P}^\dag_{l,11}){\bf e}_K+ F^\pi -BR^{-1} B^T   {\bf P}^\dag_{l,13}, \nonumber\\
& \overline {\bf G}_l^\dag(t)= G-BR^{-1} B^T{\bf P}^\dag_{l,12} \nonumber , \\
& \overline {\bf m}_l^\dag(t)= -  BR^{-1}B^T {\bf s}^\dag_{l,1} ,\label{mbars}
\end{align}
and
\begin{align}
&\overline {\bf A}^\dag(t)=
\begin{bmatrix}
\overline {\bf A}_1^\dag\\
\vdots\\
\overline {\bf A}_K^\dag
\end{bmatrix},\quad
\overline {\bf G}^\dag(t)=
\begin{bmatrix}
\overline {\bf G}_1^\dag\\
\vdots\\
\overline {\bf G}_K^\dag
\end{bmatrix}, \quad
\overline {\bf m}^\dag(t)=
\begin{bmatrix}
\overline {\bf m}_1^\dag\\
\vdots\\
\overline {\bf m}_K^\dag
\end{bmatrix}, \label{meAG}\\
&{\mathbb A}_\kappa^\dag(t) =
\begin{bmatrix}
A_\kappa & G & F^\pi\\
0_{n\times n} & A_0- B_0R_0^{-1}B_0^T {\bf P}_{0,11}^\dag & F_0^\pi - B_0R_0^{-1}B_0^T {\bf P}_{0,12}^\dag\\
0_{nK\times n} & \overline {\bf G}^\dag & \overline {\bf A}^\dag
\end{bmatrix}.  \nonumber
\end{align}

We introduce the system of ODEs:
\begin{align}
&\rho {\bf P}^\dag_0-\dot{\bf P}_0^\dag  ={\bf P}_0^\dag
\begin{bmatrix}
A_0 &F_0^\pi\\
\overline {\bf G}^\dag & \overline {\bf A}^\dag
\end{bmatrix}
+ \begin{bmatrix}
A_0 &F_0^\pi\\
\overline {\bf G}^\dag & \overline {\bf A}^\dag
\end{bmatrix}^T
{\bf P}_0^\dag -{\bf P}_0^\dag {\mathbb B}_0R_0^{-1}{\mathbb B}_0^T {\bf P}_0^\dag+{\mathbb Q}_0^\pi \label{qsP0} \\
&\rho {\bf P}_\kappa^\dag-\dot{\bf P}^\dag_\kappa  ={\bf P}_\kappa^\dag {\mathbb A}^\dag_\kappa  + {\mathbb A}^{\dag T}_\kappa {\bf P}_\kappa^\dag
 - {\bf P}_\kappa^\dag {\mathbb B}R^{-1}{\mathbb B}^T{\bf P}_\kappa^\dag+{\mathbb Q}^\pi, \quad 1\le \kappa\le K, \label{qsPk}
 \end{align}
where  $ {\bf P}_0^\dag(T)= {\mathbb  Q}_{0f}^\pi $ and ${\bf P}^\dag_\kappa(T)={\mathbb Q}^\pi_f$.
The $K+1$ equations are all coupled together through the dependence of $(\overline {\bf G}^\dag, \overline {\bf  A}^\dag, {\mathbb A}_\kappa^\dag)$ on
$({\bf P}_0^\dag  ,  {\bf P}_1^\dag, \cdots,{\bf P}_K^\dag)$.

We further introduce the ODE system:
\begin{align}
&\rho {\bf s}_0^\dag-\dot{\bf s}_0^\dag
=\begin{bmatrix}
A_0 & F_0^\pi \\
\overline {\bf G}^\dag & \overline {\bf A}^\dag
\end{bmatrix}^T {\bf s}_0^\dag
-\begin{bmatrix}
{\bf P}_{0,11}^\dag \\
{\bf P}_{0,21}^\dag
\end{bmatrix}B_0 R_0^{-1} B_0^T {\bf s}_{0,1}^\dag
+\begin{bmatrix}
{\bf P}_{0,12}^\dag \\
{\bf P}_{0,22}^\dag
\end{bmatrix}\overline {\bf m}^\dag
 -  \bar \eta_0^\pi ,\label{qss0}  \\
&\rho {\bf s}^\dag_\kappa-\dot{\bf s}^\dag_\kappa  =
({\mathbb A}_\kappa^{\dag T}\! -\! {\bf P}^\dag_\kappa {\mathbb B}R^{-1}{\mathbb B}^T) {\bf s}^\dag_\kappa
 \! -\!
 \begin{bmatrix}
{\bf  P}_{\kappa,12}^\dag\\
{\bf  P}_{\kappa,22}^\dag\\
{\bf  P}_{\kappa,32}^\dag
 \end{bmatrix}\! \! B_0 R_0^{-1}\! B_0^T {\bf s}_{0,1}^\dag
\! +\! \begin{bmatrix}
{\bf  P}_{\kappa,13}^\dag\\
{\bf  P}_{\kappa,23}^\dag\\
{\bf  P}_{\kappa,33}^\dag
 \end{bmatrix}\! \overline{\bf m}^\dag \! -\!  \bar \eta^\pi, \label{qssk}
\end{align}
where ${\bf s}_0^\dag(T)=- \bar \eta^\pi_{0f}$ and  ${\bf s}^\dag_\kappa(T)=  -\bar\eta_f^\pi$, $1\le \kappa\le K$. Note that by \eqref{meAG}, $\overline {\bf m}^\dag $ is expressed in terms of $({\bf s}_1^\dag,\cdots, {\bf s}^\dag_K)$.
\begin{theorem}
\label{theorem:me}
The system of master equations \eqref{ME1}--\eqref{ME2} has a quadratic solution \eqref{qs0}--\eqref{qsk} if and only if $({\bf P}^\dag_0, \cdots,{\bf P}^\dag_K)$ satisfies \eqref{qsP0}--\eqref{qsPk} on $[0,T]$, which further determines  $({\bf s}^\dag_0, {\bf s}^\dag_1, \cdots, {\bf s}^\dag_K)$ as a unique solution of \eqref{qss0}--\eqref{qssk} on $[0,T]$.
\end{theorem}
\proof See appendix A. \endproof

\begin{theorem}
The NCE equation system \eqref{conricK}--\eqref{conodeK} has a solution as specified in Definition \ref{def:sol} if and only if the master equation system \eqref{ME1}--\eqref{ME2} has a quadratic solution of the form \eqref{qs0}--\eqref{qsk}. If their solutions exist, they are unique and  moreover
\begin{align}
P_k ={\bf P}^\dag_k,\quad s_k={\bf s}^\dag_k,\quad   \mbox{for}\quad k=0, 1, \cdots, K. \label{PPss}
\end{align}
\end{theorem}
\proof  First of all, the NCE equation system has a solution if and only if \eqref{conricK} has a solution, which is necessarily unique.
By Theorem \ref{theorem:me}, the master equation system has a quadratic solution if and only if \eqref{qsP0}--\eqref{qsPk} has a solution.

We rewrite the ODE system of $({ P}_0, \cdots, { P}_K)$ by
expressing $(\overline A, \overline G)$ in \eqref{conricK} in terms of $({ P}_1, \cdots, { P}_K)$. The gives a new ODE system where the vector field only has the unknowns $({P}_0, { P}_1, \cdots, { P}_K)$. Subsequently we see that the above new vector field is the same as the one for $({\bf P}^\dag_0, \cdots,{\bf P}^\dag_K)$   in \eqref{qsP0}--\eqref{qsPk}. This proves the first part of the theorem together with $P_k ={\bf P}_k^\dag$ for all $0\le k\le K$. By showing the equivalence between the two ODE systems \eqref{conodeK} and  \eqref{qss0}--\eqref{qssk}, we further obtain $ s_k={\bf s}^\dag_k$ for all $0\le k\le K$. \endproof

\subsection{Comparison of feedback control laws}

Suppose \eqref{ME1}--\eqref{ME2} has a quadratic solution. By \eqref{u0hmu}--\eqref{ukhmu}
and \eqref{qs0}--\eqref{qsk},  we have
\begin{align}
&u_0(t)= -R_0^{-1}B_0^T ({\bf P}_{0,11}^\dag X_0^\dag(t) +{\bf P}^\dag_{0,12} \bar z(t)+ {\bf s}_{0,1}^\dag(t)), \label{u0q}\\
&u_\kappa(t)=-R^{-1}B^T ({\bf P}_{\kappa,11}^\dag Z^\dag_\kappa(t)+{\bf P}^\dag_{\kappa,12}X^\dag_0(t) +{\bf P}^\dag_{\kappa,13} \bar z(t)+ {\bf s}_{\kappa ,1}^\dag (t)),\label{ukq}
\end{align}
where the right hand sides use the value of the processes at time $t$.
We need to determine the equation of $\bar z$.
Under \eqref{u0q}--\eqref{ukq} we have the closed-loop equation
\begin{align}
dZ_\kappa^\dag&=( A_{\kappa}Z_\kappa^\dag +GX_0^\dag +F^\pi \bar z)dt +DdW_{\kappa}\nonumber\\
& - BR^{-1} B^T ({\bf P}_{\kappa,11}^\dag Z^\dag_\kappa+{\bf P}^\dag_{\kappa,12}X^\dag_0 +{\bf P}^\dag_{\kappa,13} \bar z + {\bf s}_{\kappa ,1}^\dag   )dt.\label{zkcl}
\end{align}
Consider $N_\kappa$ $\kappa$-type players with independent  Brownian motions and initial states of mean $\alpha_0$. We take their empirical mean
by averaging \eqref{zkcl} and let $N_\kappa\to \infty$. The limit of the empirical mean should  regenerate $\bar z$, and this derives
\begin{align}
d\bar z_\kappa 
=(& (A_\kappa -BR^{-1} B^T{\bf P}_{\kappa,11}^\dag ){\bf e}_\kappa +
F^\pi- BR^{-1} B^T {\bf P}^\dag_{\kappa,13} )\bar z dt \nonumber \\
& +(G-BR^{-1} B^T {\bf P}^\dag_{\kappa,12} ) X_0^\dag dt - BR^{-1} B^T {\bf s}_{\kappa ,1}^\dag dt,
\end{align}
where $\bar z_\kappa (0)=\alpha_0$.

By \eqref{PPss} and \eqref{zkco},
if we set $\theta_i=\kappa$, $X_\kappa^\dag(0)=\bar X_i(0) $, $X_0^\dag(0)=\bar X_0(0) $, and  $W_\kappa =W_i$,  the process $(X^\dag_\kappa, X_0^\dag , \bar z)$ under the  strategies \eqref{u0q}--\eqref{ukq} is equal to the process $(\bar X_i, \bar X_0, \bar Z)$ under the NCE based  strategies. Accordingly, the two sets of control laws are equivalent.

\section{Homogeneous Minor Players}
\label{sec:as}

For this section, all minor players form a single type so that
$A_{\theta_i}\equiv A $.
The state processes of  the $N +1$ players ${\mathcal A}_k$, $0\le k\le N$, satisfy  the  SDEs:
\begin{align}
dX_0(t)=\  &\big(A_0X_0(t)+B_0u_0(t)+F_0X^{(N)}(t)\big)dt
 +D_0dW_0(t),\label{stateX0}\\
dX_i(t)=\  &\big(AX_i(t)+Bu_i(t)+FX^{(N)}(t)+GX_0(t)\big)dt
+DdW_i(t),\label{stateXi} \\
& 1\le i\le N,\quad t\ge 0,  \nonumber
\end{align}
The costs are given by \eqref{costJ0}--\eqref{costJi}.

\subsection{The Nash certainty equivalence equation system}

We follow the notation in Section \ref{sec:sub:nce} to denote
\begin{align*}
&\mathbb{A}_0(t)=
\begin{bmatrix}
A_0 & F_0\\
\overline G(t) & \overline A (t)
\end{bmatrix}, \quad
\mathbb{B}_0=
\begin{bmatrix}
B_0 \\
0_{n\times n_1}
\end{bmatrix},\quad
\mathbb{M}_0(t)=
\begin{bmatrix}
0_{n\times 1}\\
\overline m(t)
\end{bmatrix}\\
& \mathbb{Q}_0= [I, -\Gamma_0]^T Q_0 [I, -\Gamma_0], \quad
\bar{\eta}_0=[I, -\Gamma_0]^T Q_0 \eta_0\\
& \mathbb{Q}_{0f}= [I, -\Gamma_{0f}]^T Q_{0f} [I, -\Gamma_{0f}], \quad
\bar{\eta}_{0f}=[I, -\Gamma_{0f}]^T Q_{0f} \eta_{0f},
\end{align*}
where we have  $\overline A(t) $, $\overline G(t)\in \mathbb{R}^{n\times n} $,  $\overline m(t)\in \mathbb{R}^n $.
 Denote
\begin{align*}
&\mathbb{A}(t)=
\begin{bmatrix}
A   &  [G\qquad F]  \\
0_{2n\times n}  &  \mathbb{A}_0(t)-\mathbb{B}_0 R_0^{-1} \mathbb{B}_0^T P_0(t)
\end{bmatrix} ,\quad
\mathbb{B} =
\begin{bmatrix}
B\\
0_{2n\times n_1}
\end{bmatrix}, \\
&\mathbb{M}(t)=
\begin{bmatrix}
0_{n\times 1}\\
\mathbb{M}_0- \mathbb{B}_0 R_0^{-1} \mathbb{B}_0^Ts_0
\end{bmatrix} ,  \\
& \mathbb{Q}=[I, -\Gamma_1, -\Gamma_2]^T Q[I, -\Gamma_1, -\Gamma_2], \qquad \bar \eta = [I, -\Gamma_1, -\Gamma_2]^T Q \eta,\\
& \mathbb{Q}_f=[I, -\Gamma_{1f}, -\Gamma_{2f}]^T Q_f[I, -\Gamma_{1f}, -\Gamma_{2f}], \qquad \bar \eta_f = [I, -\Gamma_{1f}, -\Gamma_{2f}]^T Q_f \eta_f.
\end{align*}

The NCE equation system   \eqref{conricK}--\eqref{conodeK} now reduces to  i)
\begin{align}
\left\{
\begin{array}{l}
\rho P_0= \dot{P}_0+ P_0\mathbb{A}_0+\mathbb{A}_0^T P_0 -P_0 \mathbb{B}_0 R_0^{-1}\mathbb{B}_0^T P_0 +{\mathbb Q}_0,\\
\rho P_1 =\dot{P}_1+P_1 \mathbb{A} +\mathbb{A}^T
P_1 - P_1 {\mathbb{B}} R^{-1}{\mathbb{B}}^T P_1 +\mathbb{Q},\\
  \overline A(t)= A + F- BR^{-1} B^T P_{1, 11} - BR^{-1} B^TP_{1, 13} ,  \\
 \overline G(t) =G-BR^{-1} B^T P_{1, 12} ,
\end{array}\right. \label{conric}
\end{align}
where
${ P}_0(T)= \mathbb{Q}_{0f}$,
${ P}_1(T) = {\mathbb Q}_f$,
\begin{align}
P_1=(P_{1,jk})_{1\le j,k\le 3}
\quad P_{1,jk}(t)\in \mathbb{R}^{n\times n},
 \notag
\end{align}
and ii)
\begin{align}
\left\{
\begin{array}{l}
\rho s_0= \dot{s}_0 +(\mathbb{A}_0^T-P_0\mathbb{B}_0 R_0^{-1} \mathbb{B}_0^T) s_0+P_0 \mathbb{M}_0-\bar\eta_0,
\\
\rho  s_1= \dot{s}_1 +(\mathbb{A}^T-P_1 {\mathbb{B}} R^{-1} {\mathbb{B}}^T)  s_1+P_1 {\mathbb{M}}-\bar\eta, \\
\overline m(t) =-   BR^{-1} {\mathbb{B}}^T s_1(t),
\end{array}\right.\label{conode}
\end{align}
where ${ s}_0(T) =- \bar\eta_{0f}$,
$ s_1(T)= -\bar\eta_f$.

\subsection{The asymptotic solvability problem}

To begin with, we provide some background on asymptotic solvability  based on \cite{MH20}.
Denote by ${\bf 1}_{k\times l}$  a $k\times l$ matrix with all entries equal to 1, and by the column  vectors  $\{e_1^k, \ldots, e_k^k\}$  the canonical basis of $\mathbb {R}^k$. For instance, $e_1^k=[1, 0 ,\cdots, 0]^T\in \mathbb{R}^k$.
  Define
\begin{align}
&\mbx=(\mbx_0^T,\mbx_1^T,\cdots,\mbx_N^T)^T\in\mathbb{R}^{(N+1)n},
\nonumber\\
&X(t)= \begin{bmatrix} X_0(t) \\ \vdots \\ X_N(t)
\end{bmatrix}\in \mathbb{R}^{(N+1)n},\nonumber
\quad W(t)=\begin{bmatrix}W_0(t) \\ \vdots \\ W_N(t)
\end{bmatrix}\in \mathbb{R}^{(N+1)n_2},\nonumber   \\
&\widehat \mbA = \mbox{diag}[A_0, A, \cdots, A]+\left[
\begin{matrix}
& 0_{n\times n},&{\bf 1}_{1\times N}\otimes \frac{F_0}{N}\\
& {\bf 1}_{N\times 1}\otimes G, &{\bf 1}_{N\times N}\otimes
\frac{F}{N}
\end{matrix}\right],\nonumber\\
&\widehat \mbD =\mbox{diag}[D_0, D, \cdots, D]\in \mathbb{R}^{(N+1)n\times (N+1)n_2},\nonumber \\
&\mbB_0 = e_1^{N+1}\otimes B_0\in\mathbb{R}^{(N+1)n\times n_1}, \nonumber\\
&\mbB_{k} = e_{k+1}^{N+1}\otimes B\in\mathbb{R}^{(N+1)n\times n_1},\qquad 1 \leq k \leq N. \nonumber
\end{align}

Now we write \eqref{stateX0} and \eqref{stateXi} in the form
\begin{align}\label{bigx}
dX(t)=\Big(\widehat \mbA X(t)+\sum_{k=0}^N \mbB_k u_k(t)\Big)dt+\widehat \mbD dW(t), \quad t\ge 0.
\end{align}
We consider closed-loop perfect state (CLPS) information \cite{BO99} so that
$X(t)$ is observed by each player, and look for Nash strategies. Let $u_{-k}$ denote the strategies of all players other than ${\mathcal A}_k$.
 A set of strategies
 $(\hat u_0,\cdots, \hat u_N)$ is a Nash equilibrium
if for all $0\le k\le N$,  we have
$
J_k(\hat u_k, \hat u_{-k})\le J_k(u_k, \hat u_{-k}),
$
for any state feedback based strategy $u_k$ which together with $\hat u_{-k}$ ensures a unique solution of $X(t)$ on $[0,T]$.
Denote
\begin{align}
&\mbK_{0}= [{I_n},0,\cdots,0]-\tfrac{1}{N}
[0,\Gamma_{0},\cdots,\Gamma_{0}], \nonumber  \\
& \mbK_{0f}= [{I_n},0,\cdots,0]-\tfrac{1}{N}
[0,\Gamma_{0f},\cdots,\Gamma_{0f}], \notag \\
& \mbQ_{0}=\mbK_{0}^T Q_{0} \mbK_{0}, \qquad  \mbQ_{0f}=\mbK_{0f}^T Q_{0f} \mbK_{0f}, \notag\\
&\mbK_{i}=[0,0,\cdots,{I}_{n},0,\cdots,0]-[\Gamma_{1},0,\cdots,0]  -\tfrac{1}{N}[0,\Gamma_{2},\cdots,\Gamma_{2}], \label{kig}\\
&\mbK_{if}=[0,0,\cdots,{I}_{n},0,\cdots,0]-[\Gamma_{1f},0,\cdots,0]  -\tfrac{1}{N}[0,\Gamma_{2f},\cdots,\Gamma_{2f}], \notag  \\
& \mbQ_{i}=\mbK_{i}^T Q \mbK_{i},  \qquad \mbQ_{if}=\mbK_{if}^T Q_f \mbK_{if},\notag\\
&\widehat \mbA_{\rho/2}= \widehat \mbA -\rho I/2,\qquad \widehat \mbA_{\rho}= \widehat \mbA -\rho I , \notag
\end{align}
where $I_n$ is the $(i+1)$th submatrix in \eqref{kig}. We have $\mbK_0,\mbK_i\in \mathbb {R}^{n\times(N+1)n}$ and $\mbQ_0,\mbQ_i\in \mathbb {R}^{(N+1)n\times(N+1)n}$.

Based on \cite{MH20},    we  introduce the equation system:
\begin{align}\label{DE3_P0}
\begin{cases}
\dot{\mbP}_0(t) =  - \big({\mbP}_0\widehat{\mbA}_{\rho/2}+\widehat{\mbA}_{\rho/2}^T \mathbb{\mbP}_0\big)+
                   {\mbP}_0{\mbB}_0 R_0^{-1} {\mbB}_0^T{\mbP}_0 \\
\qquad\qquad        + {\mbP}_0\sum_{k=1}^N {\mbB}_k R^{-1} {\mbB}_k^T {\mbP}_k
         +\sum_{k=1}^N {\mbP}_k{\mbB}_k R^{-1} {\mbB}_k^T {\mbP}_0 -{\mbQ}_{0}   , \\
 {\mbP}_0(T) = {\mbQ}_{0f},
 \end{cases}
\end{align}
\begin{align}\label{DE3_S0}
\begin{cases}
\dot{{\mbS}}_0(t) = - \widehat{{\mbA}}^T_{\rho} {\mbS}_0  +  {\mbP}_0{\mbB}_0 R_0^{-1} {\mbB}_0^T {\mbS}_0
       + \sum_{k=1}^N{\mbP}_k {\mbB}_k R^{-1} {\mbB}_k^T {\mbS}_0\\
\qquad\qquad        + {\mbP}_0\sum_{k=1}^N {\mbB}_k R^{-1}{\mbB}_k^T {\mbS}_k
   +\mbK_0^T Q_{0}\eta_{0} , \\
{\mbS}_0(T)= -\mbK_{0f}^T Q_{0f}\eta_{0f},
\end{cases}
\end{align}
\begin{align}\label{DE3_P}
\begin{cases}
\dot{\mbP}_i(t) =  - \big({\mbP}_i\widehat{\mbA}_{\rho/2}+\widehat{\mbA}^T_{\rho/2} \mathbb{\mbP}_i\big)
-{\mbP}_i{\mbB}_i R^{-1} {\mbB}_i^T{\mbP}_i\\
    \qquad\qquad   +  \big({\mbP}_i{\mbB}_0 R_0^{-1} {\mbB}_0^T{\mbP}_0+
                           {\mbP}_0{\mbB}_0 R_0^{-1} {\mbB}_0^T{\mbP}_i\big)\\
    \qquad\qquad   + \Big({\mbP}_i\sum_{k=1}^N {\mbB}_k R^{-1} {\mbB}_k^T {\mbP}_k+
            \sum_{k=1}^N {\mbP}_k{\mbB}_k R^{-1} {\mbB}_k^T {\mbP}_i\Big) -{\mbQ}_{i} ,\\
 {\mbP}_i(T) = {\mbQ}_{if},\qquad 1\leq i \leq N,
 \end{cases}
\end{align}
\begin{align}\label{DE3_S}
\begin{cases}
\dot{{\mbS}}_i(t) = - \widehat{{\mbA}}^T_{\rho} {\mbS}_i  +  {\mbP}_0{\mbB}_0 R_0^{-1} {{\mbB}_0}^T {\mbS}_i\\
          \qquad\qquad    + {\mbP}_i{\mbB}_0 R_0^{-1} {\mbB}_0^T {\mbS}_0    
            -{\mbP}_i{\mbB}_i R^{-1} {\mbB}_i^T {\mbS}_i\\
          \qquad\qquad    + \sum_{k=1}^N{\mbP}_k {\mbB}_k R^{-1} {\mbB}_k^T {\mbS}_i
              + {\mbP}_i\sum_{k=1}^N {\mbB}_k R^{-1}{\mbB}_k^T {\mbS}_k   
              +\mbK_{i}^T Q \eta , \\
{\mbS}_i(T)= -\mbK_{if}^T Q_{f}\eta_{f}, \qquad 1\leq i \leq N.
\end{cases}
\end{align}
The analysis in \cite{MH20} is for costs without discount.
The notion of asymptotic solvability and main results in \cite{MH20}
 can be translated to the discounted case verbatim once we let $\widehat {\mbA}_{\rho/2}$ take the role of $\widehat {\mbA}$ used in \cite{MH20} for Riccati ODEs.

Suppose that \eqref{DE3_P0} and \eqref{DE3_P} have a unique solution $(\mbP_0,\cdots ,\mbP_N)$ on $[0,T]$. Then we can uniquely solve \eqref{DE3_S0}, \eqref{DE3_S}, and the Nash game of $N+1$ players has a set of feedback Nash strategies given by
 \begin{align*}
&\hat u_0(t)=-R_0^{-1}\mbB_0^{T}(\mbP_0X(t)+\mbS_0), \\ 
&\hat u_i(t)=-R^{-1} \mbB_i^T (\mbP_i X(t) +\mbS_i),\qquad  1\le i\le N. 
 \end{align*}
  The solution of the feedback Nash strategies completely reduces to the study of \eqref{DE3_P0} and \eqref{DE3_P}.

\begin{definition} \emph{\cite{MH20}} \label{definition:as}
The sequence of  Nash games specified by \eqref{stateX0}--\eqref{stateXi} and \eqref{costJ0}--\eqref{costJi} has asymptotic solvability if there exists $N_0$ such that for all $N\ge N_0$, the Riccati ODE system consisting of \eqref{DE3_P0} and \eqref{DE3_P} has a solution
 $(\mbP_0,\cdots,\mbP_N)$  on $[0,T]$ and
$
 \sup_{N\ge N_0, 0\leq t\leq T} \left({\|\mbP_0(t)\|}_{l_1}+
{\|\mbP_1(t)\|}_{l_1}\right)<\infty.
$
 \end{definition}

 Denote
 $$M_0=B_0R_0^{-1}B_0^T,\qquad M=BR^{-1}B^T.$$
We introduce  the  ODE system:
\begin{align}
\left\{
\begin{array}{l}
 \dot{\Lambda}_1^{0} =\rho{\Lambda}_1^{0}  + \Lambda_1^{0}M_0 \Lambda_1^{0}-(\Lambda_1^{0}A_0+A_0^T\Lambda_1^{0}) \\
\quad\quad +\Lambda_2^{0}(M\Lambda_a^T-G)+(\Lambda_a M-G^T){\Lambda_2^{0}}^T-Q_0, \\
\dot{\Lambda}_2^{0} =\rho {\Lambda}_2^{0}  + ( \Lambda_1^{0}M_0 -A_0^T) \Lambda_2^{0}+\Lambda_2^{0}(M(\Lambda_1+\Lambda_2)-A-F) \\
\quad\quad-\Lambda_1^{0}F_0+ (\Lambda_a M-G^T)\Lambda_3^{0}+Q_0\Gamma_0,\\
\dot{\Lambda}_3^{0} =\rho {\Lambda}_3^{0}+ {\Lambda_2^{0}}^TM_0 \Lambda_2^{0}-{\Lambda_2^{0}}^TF_0-F_0^T\Lambda_2^{0} \\
\quad\quad +\Lambda_3^{0}(M(\Lambda_1+\Lambda_2)-A-F) \\
\quad\quad +((\Lambda_1+\Lambda_2^T)M-A^T -F^T)\Lambda_3^{0}-\Gamma_0^TQ_0\Gamma_0,\\
 \dot{\Lambda}_0    =\rho {\Lambda}_0+  \Lambda_a M\Lambda_a^T-\Lambda_bG-G^T \Lambda_b^T \\
\quad\quad +\Lambda_0(M_0\Lambda_1^{0}-A_0)+(\Lambda_1^{0} M_0-A_0^T) \Lambda_0 \\
\quad\quad -\Lambda_a(G-M \Lambda_b^T)-(G^T-\Lambda_bM)
\Lambda_a^T-\Gamma_1^TQ\Gamma_1, \\
\dot{\Lambda}_1 =  \rho{\Lambda}_1  +\Lambda_1M\Lambda_1-\Lambda_1A-A^T\Lambda_1-Q,\\
 \dot{\Lambda}_2 = \rho{\Lambda}_2 +\Lambda_a^T(M_0\Lambda_2^{0}-F_0)-\Lambda_1F+(\Lambda_1M-A^T)\Lambda_2 \\
\quad\quad+\Lambda_2(M(\Lambda_1+\Lambda_2)-A-F)+Q\Gamma_2,\\
\dot{\Lambda}_3 = \rho {\Lambda}_3   + \Lambda_b^TM_0\Lambda_2^{0}+{\Lambda_2^{0}}^TM_0\Lambda_b+\Lambda_2^TM\Lambda_2 \\
\quad\quad -\Lambda_b^TF_0-F_0^T\Lambda_b-\Lambda_2^TF-F^T\Lambda_2  \\
\quad\quad +\Lambda_3(M(\Lambda_1+\Lambda_2)-A-F) \\
\quad\quad +((\Lambda_1+\Lambda_2^T) M-A^T-F^T)\Lambda_3- \Gamma_2^T Q\Gamma_2,\\
\dot{\Lambda}_a = \rho {\Lambda}_a   + (\Lambda_1^{0}M_0-A_0^T)\Lambda_a+\Lambda_a(M\Lambda_1-A) \\
\quad\quad-G^T\Lambda_1+(\Lambda_aM-G^T)\Lambda_2^T+\Gamma_1^T Q,\\
 \dot{\Lambda}_b = \rho {\Lambda}_b + \Lambda_0M_0\Lambda_2^{0}+(\Lambda_aM-G^T)(\Lambda_2+\Lambda_3)-\Lambda_0 F_0 \\
\quad\quad
-\Lambda_a F  +\Lambda_b(M(\Lambda_1+\Lambda_2)-A-F)\\
\quad\quad  +(\Lambda_1^{0}M_0-A_0^T)\Lambda_b -\Gamma_1^T Q\Gamma_2,
\end{array}
\right.\label{eqn1*}
\end{align}
where the terminal conditions are
\begin{align}
\begin{cases}
 \Lambda_1^0(T)=Q_{0f},\quad
 \Lambda_2^0(T)=-Q_{0f}\Gamma_{0f}, \quad
 \Lambda_3^0(T)=\Gamma_{0f}^TQ_{0f}\Gamma_{0f},\nonumber  \\
\Lambda_0 (T)=\Gamma_{1f}^TQ_f\Gamma_{1f},\quad
\Lambda_1 (T)=Q_{f},\nonumber \\
  \Lambda_2 (T)=- Q_f\Gamma_{2f},\quad
   \Lambda_3 (T)=\Gamma_{2f}^T Q_f\Gamma_{2f},\nonumber\\
 \Lambda_a (T)=-\Gamma_{1f}^T Q_f ,\quad
\Lambda_b (T)=\Gamma_{1f}^T Q_f\Gamma_{2f}.\nonumber  \\
\end{cases}
\end{align}

To give the reader some insights, we explain from where the equations in \eqref{eqn1*} arise. Note that asymptotic solvability is stated in terms of $(\mbP_0, \mbP_1, \cdots, \mbP_K)$, the dimension of which increases with the population size.
The discount factor having been absorbed in  $\widehat \mbA_{\rho/2}$, we follow  the procedure in \cite{MH20} to isolate a low dimensional structure from $(\mbP_0, \mbP_1, \cdots, \mbP_K)$.
Specifically, by exploiting symmetry of the ODEs \eqref{DE3_P0} and \eqref{DE3_P}, it can be shown that the large matrix $\mbP_0$ is formed from  3
distinct $n\times n$ submatrices by arranging them into $N^2$ places. Similarly, $\mbP_1$ is formed from 6 distinct $n\times n$ submatrices. Any other matrix $\mbP_k$, $k\ge 2$, is determined from $\mbP_1$ by appropriate simultaneous row and column permutations. By using the above 9 submatrices and applying appropriate re-scaling to  individual submatrices, we  derive the  ODE system \eqref{eqn1*} as $N\to \infty$; see \cite{MH20} for more details.

\begin{theorem}  {\emph{\cite{MH20}}}
\label{theorem:AS}
The sequence of games with dynamics  \eqref{stateX0}--\eqref{stateXi} and costs \eqref{costJ0}--\eqref{costJi} has asymptotic solvability if and only if \eqref{eqn1*} has a  solution on $[0,T]$.
\end{theorem}

\begin{theorem}\label{theorem:as}
The NCE equation system \eqref{conric}--\eqref{conode}  has a solution on $[0,T]$ if and only if asymptotic solvability holds.
\end{theorem}

\begin{proof}
By Lemma \ref{lemma:NCE}, \eqref{conric}--\eqref{conode} has a solution if and only if \eqref{conric} has a solution on $[0,T]$; in addition, $P_0$ and $P_1$ in such a solution are symmetric.
We denote the two matrix functions $P_0$ and $P_1$
in \eqref{conric} in the form
\begin{align}
P_0(t)=
\begin{bmatrix}
\Phi_1^0 & \Phi_2^0  \\
\Phi_2^{0T} & \Phi_3^0  \\
\end{bmatrix}, \quad P_1(t)=
\begin{bmatrix}
\Phi_1 & \Phi_a^T &\Phi_2 \\
\Phi_a & \Phi_0 &\Phi_b \\
\Phi_2^T & \Phi_b^T &\Phi_3 \\
\end{bmatrix}, \nonumber
\end{align}
where each submatrix is $n\times n$.
For ${\mathbb A}_0$ and $\mathbb{A}$ in \eqref{conric},
we rewrite
\begin{align*}
&{\mathbb A}_0(t) =\begin{bmatrix}
A_0 & F_0\\
 G-M\Phi_a^T &  A +F-M(\Phi_1+\Phi_2)
\end{bmatrix}, \\
& \mathbb{A}(t)=
\begin{bmatrix}
A   &  G & F  \\
0_{n\times n}  &  A_0-M_0 \Phi_1^0  & F_0-M_0\Phi_2^0 \\
0_{n\times n}&  G-M\Phi_a^T  &A +F-M(\Phi_1+\Phi_2)
\end{bmatrix}.
\end{align*}

From \eqref{conric}
we obtain the following ODE system
\begin{align}
\begin{cases}
\dot\Phi_1^0= \rho \Phi_1^0 - \Phi_1^0 A_0 -A_0^T \Phi_1^0 -  \Phi_2^0 (G-M \Phi_a^T)\\
 \qquad -(G^T-\Phi_a M)\Phi_2^{0T} +\Phi_1^0M_0\Phi_1^0-Q_0,\\
 \dot \Phi_2^0   =  \rho \Phi_2^0-\Phi_1^0F_0  - \Phi_2^0(A+F-
 M(\Phi_1+\Phi_2) )       \\
 \qquad -A_0^T\Phi_2^0 +(\Phi_aM-G^T) \Phi_3^0   +\Phi_1^0 M_0 \Phi_2^0 +Q_0 \Gamma_0,\\
\dot \Phi_3^0=  \rho \Phi_3^0 -\Phi_2^{0T} F_0 - F_0^T \Phi_2^0 -\Phi_3^0(A+F-
 M(\Phi_1+\Phi_2) )\\
 \qquad  -(A+F-
 M(\Phi_1+\Phi_2) )^T \Phi_3^0 +\Phi_2^{0T} M_0 \Phi_2^0 -\Gamma_0^T Q_0 \Gamma_0 , \\
 \dot \Phi_0=\rho \Phi_0 -\Phi_a G - \Phi_0 (A_0-M_0 \Phi_1^0) -\Phi_b (G-M \Phi_a^T) -G^T \Phi_a^T\\
\qquad -(A_0^T-\Phi_1^0M_0)\Phi_0
- (G^T- \Phi_a M) \Phi_b^T+\Phi_a M \Phi_a^T-\Gamma_1^TQ\Gamma_1,\\
 \dot \Phi_1 = \rho \Phi_1 -\Phi_1 A-A^T \Phi_1 +\Phi_1 M \Phi_1-Q,   \\
 \dot \Phi_2 =  \rho \Phi_2 -\Phi_1 F -\Phi_a^T (F_0-M_0 \Phi_2^0)-\Phi_2 (A+F- M(\Phi_1+\Phi_2) ) \\
\qquad - A^T \Phi_2 +\Phi_1M \Phi_2 +Q\Gamma_2 ,\\
\dot \Phi_3  =\rho \Phi_3 -\Phi_2^T F-F^T\Phi_2
-\Phi_b^T(F_0-M_0 \Phi_2^0)-(F_0^T-\Phi_2^{0T}M_0) \Phi_b\\
\qquad -\Phi_3 (A+F- M(\Phi_1+\Phi_2) )-(A+F- M(\Phi_1+\Phi_2) )^T \Phi_3\\
 \qquad  +\Phi_2^T M \Phi_2  -\Gamma_2^T Q \Gamma_2,\\
\dot \Phi_a = \rho \Phi_a -\Phi_aA -G^T \Phi_1 - (A_0^T - \Phi_1^0 M_0)
\Phi_a \\
 \qquad -(G^T- \Phi_a M) \Phi_2^T +\Phi_a M \Phi_1 +\Gamma_1^TQ,   \\
\dot \Phi_b = \rho \Phi_b -\Phi_a F-\Phi_0 (F_0-M_0 \Phi_2^0)-\Phi_b  (A+F- M(\Phi_1+\Phi_2) )\\
\qquad -G^T \Phi_2-(A_0^T-\Phi_1^0 M_0) \Phi_b -(G^T- \Phi_a M) \Phi_3\\
\qquad +\Phi_a M \Phi_2 -\Gamma_1^T Q \Gamma_2,
\end{cases} \label{conricPhi}
\end{align}
where the terminal conditions are given by
\begin{align}
\begin{cases}
\Phi_1^0(T)=Q_{0f},   \quad \Phi_2^0(T)=-Q_{0f}\Gamma_{0f},   \quad \Phi_3^0(T)=
\Gamma_{0f}^TQ_{0f}\Gamma_{0f},\\
 \Phi_0(T)=\Gamma_{1f}^TQ_f\Gamma_{1f}, \quad \Phi_1(T)=  Q_f, \\
 \Phi_2(T)=-Q_f\Gamma_{2f}, \quad \Phi_3(T)=\Gamma_{2f}^TQ_f\Gamma_{2f},\\
  \Phi_a(T)=-\Gamma_{1f}^TQ_f, \quad \Phi_b(T)=\Gamma_{1f}^TQ_f\Gamma_{2f}. \end{cases}\nonumber
\end{align}
By comparing the individual equations at the corresponding place of \eqref{eqn1*} and \eqref{conricPhi}, we see the two equation systems are determined by the same vector field with the same terminal conditions, and therefore they have the same solution.
In view of  Theorem \ref{theorem:AS} and Lemma \ref{lemma:NCE}, the theorem follows.
\end{proof}

\begin{remark}
The equation system \eqref{eqn1*} originates from the ODE system of
${\mbP_0}$, $\mbP_1$, $\cdots$, $\mbP_K$. But \eqref{conricPhi}, which is equivalent to \eqref{eqn1*}, arises from solving two optimal control problems in a low dimensional space as in Section \ref{sec:sub:nce}.
\end{remark}

\section{Conclusion}
\label{sec:con}

This paper considers LQ mean field games with a major player and investigates the relationship between several solution frameworks.
For a model of minor players of several subpopulations, an equivalence relationship is established between the Nash certainty equivalence approach and master equations. For a model with homogeneous minor players, it is shown that the Nash certainty equivalence based solution exists if and only if asymptotic solvability holds.

\section*{Appendix A:  Proof of Theorem  \ref{theorem:me} }
\renewcommand{\theequation}{A.\arabic{equation}}
\setcounter{equation}{0}
\renewcommand{\thetheorem}{A.\arabic{theorem}}
\setcounter{theorem}{0}
\renewcommand{\thesection}{A.\arabic{section}}
\setcounter{section}{0}

Note that if $({\bf P}^\dag_0, \cdots,{\bf P}^\dag_K)$ is a solution of \eqref{qsP0}--\eqref{qsPk}  on $[0,T]$, it is the unique solution by the local Lipschitz continuity of the vector fields in the $K+1$ matrix ODEs. By \eqref{mbars},
$\overline {\bf m}^\dag$ depends linearly on $({\bf s}^\dag_1, \cdots, {\bf s}^\dag_K)$.
If  $({\bf P}^\dag_0, \cdots,{\bf P}^\dag_K)$ exists on $[0,T]$, we may uniquely solve $({\bf s}^\dag_0, {\bf s}^\dag_1, \cdots, {\bf s}^\dag_K)$ from a system of linear ODEs with bounded coefficients.

We now take \eqref{qs0}--\eqref{qsk} defined for $t\in [0,T]$ as a candidate solution of the master equations.
Our plan is to substitute $(V_0, \cdots, V_K)$ into the right hand side of each of the $K+1$ equations in \eqref{ME1}--\eqref{ME2} and simplify the expression into a  quadratic form of $\xi_0$ or $\xi_\kappa$.
We directly compute the derivatives:
\begin{align*}
&\partial_{x_0} V_0 =2[{\bf P}_{0,11}^\dag, {\bf P}_{0,12}^\dag] \xi_0 +2{\bf s}^\dag_{0,1},\qquad \partial_{x_0x_0} V_0=2{\bf P}_{0,11}^\dag ,\\
&\partial_{x_0} V_\kappa = 2[{\bf P}_{\kappa,21}^\dag,{\bf P}_{\kappa,22}^\dag,{\bf P}_{\kappa,23}^\dag  ]\xi_{\kappa} +2 {\bf s}_{\kappa, 2}^\dag,\qquad
\partial_{x_0x_0} V_\kappa = 2{\bf P}_{\kappa,22}^\dag,\\
& \partial_{z_\kappa} V_\kappa = 2[{\bf P}_{\kappa,11}^\dag,{\bf P}_{\kappa,12}^\dag,{\bf P}_{\kappa,13}^\dag  ]\xi_{\kappa}+ 2 {\bf s}_{\kappa,1}^\dag,\qquad
 \partial_{z_\kappa z_\kappa } V_\kappa = 2{\bf P}_{\kappa,11}^\dag.
\end{align*}


To facilitate the subsequent computation, we state two lemmas involving derivatives with respect to probability measures. Recall the notation in \eqref{dmuv0v}.
\begin{lemma}\label{lemma:Vdmu}
Suppose $V_0$ and $V_\kappa$ are given by \eqref{qs0}--\eqref{qsk}. Then we have
\begin{align*}
&[\partial_{y_1}^T\partial_{\mu_1} V_0, \cdots, \partial_{y_K}^T\partial_{\mu_K} V_0]=2\xi_0^T
\begin{bmatrix}
{\bf P}_{0,12}^\dag\\
{\bf P}_{0,22}^\dag
\end{bmatrix}
+2{\bf s}_{0,2}^{\dag T},\\
& [\partial_{y_1}^T\partial_{\mu_1} V_\kappa, \cdots, \partial_{y_K}^T\partial_{\mu_K} V_\kappa]=2(z_\kappa^T {\bf P}^\dag_{\kappa, 13} + x_0^T{\bf P}^\dag_{\kappa, 23}+  \bar z^T{\bf P}^\dag_{\kappa, 33})+2{\bf s}^{\dag T}_{\kappa,3}.
\end{align*}
\end{lemma}
\proof
Let ${\bf P}_{0,22}^{^\dag{\rm col}1}$ denote the sub-matrix consisting of the first $n$ columns of ${\bf P}^\dag_{0,22}$. Let
${\bf P}_{0,21}^{^\dag{\rm row}1}$ denote the submatrix consisting of the first $n$ rows of ${\bf P}^\dag_{0,21}$, and
${\bf s}_{0,2}^{^\dag{\rm row} 1}$ be the subvector of the first $n$ entries of ${\bf s}^\dag_{0,2}$. We obtain
\begin{align}
\partial_{\mu_1} V_0 = 2\bar z^T {\bf P}_{0,22}^{\dag{\rm col}1} y_1 + 2y_1^T{\bf P}^{\dag{\rm row}1}_{0,21} x_0 + 2y_1^T {\bf s}_{0,2}^{\dag{\rm row} 1}.\notag
\end{align}
Therefore
\begin{align}
&\partial_{y_1}^T\partial_{\mu_1} V_0 = 2\bar z^T {\bf P}_{0,22}^{\dag{\rm col}1}+
2x_0^T{\bf P}_{0,21}^{\dag{\rm row1}T}+2{\bf s}_{0,2}^{\dag{\rm row 1}T}.\notag
\end{align}
We similarly calculate $\partial_{y_l}^T\partial_{\mu_l} V_0$ for $l\ge 2$,
and obtain the first equality in the lemma.
Next we have
\begin{align}
\partial_{\mu_l} V_\kappa = \partial_{\mu_l} (\bar z^T {\bf P}^\dag_{\kappa, 33} \bar z +2 \bar z^T {\bf P}^\dag_{\kappa, 31} z_\kappa + 2\bar z^T {\bf P}^\dag_{\kappa, 32} x_0  +2\bar z^T {\bf s}^\dag_{\kappa,3}). \notag
\end{align}
Let ${\bf P}_{\kappa, 33}^{\dag{\rm col}1}$ be the submatrix consisting of the first $n$ columns of ${\bf P}^\dag_{\kappa, 33}$. Then
$
\partial_{\mu_1}  (\bar z^T {\bf P}^\dag_{\kappa, 33} \bar z)=2 \bar z^T{\bf P}_{\kappa, 33}^{\dag{\rm col}1} y_1,
$
which gives
$\partial_{y_1}^T\partial_{\mu_1}  (\bar z^T {\bf P}^\dag_{\kappa, 33} \bar z)=2 \bar z^T{\bf P}_{\kappa, 33}^{\dag{\rm col}1}$.
We further obtain the second equality.
The lemma follows.
\endproof

\begin{lemma}\label{lemma:dmuyy}
We have $\partial_{y_ly_l}\partial_{\mu_l} V_0=0$ and
$\partial_{y_ly_l}\partial_{\mu_l} V_\kappa =0$  for all $1\le l\le K$.

\end{lemma}

\proof This follows from Lemma  \ref{lemma:Vdmu}. \endproof

We proceed to evaluate the right hand side of  \eqref{ME1} with the candidate solution $(V_0, \cdots, V_K)$. We have
\begin{align*}
&\chi_{0,1}=\partial_{x_0}^T V_0(A_0 x_0+F_0^\pi \bar z)=2(\xi_0^T [{\bf P}^\dag_{0,11}, {\bf P}^\dag_{0,12}]^T+ {\bf s}_{0,1}^{\dag T})(A_0 x_0+F_0^\pi  \bar z),\\
&\chi_{0,2}= |[{\bf P}_{0,11}^\dag, {\bf P}^\dag_{0,12}]\xi_0+ {\bf s}^\dag_{0,1}|^2_{B_0R_0^{-1} B_0^T},\\
& \chi_{0,3}=| [I, -\Gamma_0^\pi] \xi_0-\eta_0|_{Q_0}^2 ,  \\
&\chi_{0,4}= \mbox{Tr} ({\bf P}^\dag_{0,11} D_0D_0).
\end{align*}
Since $\partial_{y_l} V_l(t, y_l, x_0, \mu) = 2[{\bf P}^\dag_{l,11}, {\bf P}^\dag_{l,12}, {\bf P}^\dag_{l,13}]\xi_l|_{z_l=y_l} +2{\bf s}^\dag_{l,1} $,
we further write
\begin{align}
&A_l y_l +Gx_0+F^\pi \bar z-({1}/{2}) BR^{-1}B^T \partial_{y_l} V_l(t, y_l, x_0, \mu)\notag\\
= \ & (G-BR^{-1} B^T{\bf P}^\dag_{l,12})x_0+
 (A_l- BR^{-1} B^T{\bf P}^\dag_{l,11})y_l \notag\\
 &+ ( F^\pi -BR^{-1} B^T   {\bf P}^\dag_{l,13})\bar z -  BR^{-1}B^T {\bf s}^\dag_{l,1}. \label{Aly}
\end{align}
By \eqref{meAG}, \eqref{Aly} and Lemma \ref{lemma:Vdmu},
\begin{align}
\chi_{0,5}-\chi_{0,6}&=[\partial_{y_1}^T\partial_{\mu_1} V_0, \cdots, \partial_{y_K}^T\partial_{\mu_K} V_0]([\overline {\bf G}^\dag, \overline {\bf A}^\dag ]\xi_0 +\overline
{\bf m}^\dag)\notag \notag\\
&=2(\xi_0^T
 \begin{bmatrix}
{\bf P}^\dag_{0,12}\\
{\bf P}^\dag_{0,22}
\end{bmatrix}+{\bf s}_{0,2}^{\dag T})
([\overline {\bf G}^\dag, \overline {\bf A}^\dag ]\xi_0+\overline {\bf m}^\dag).
\end{align}
By Lemma \ref{lemma:dmuyy},
$
\chi_{0,7}=0.
$

We further evaluate the right hand side of \eqref{ME2}.
We have
\begin{align*}
\chi_1-\chi_2&=\partial_{x_0}^T V_\kappa(A_0 x_0  +F_0^\pi \bar z )- ({1}/{2})\partial_{x_0}^T V_\kappa B_0 R_0^{-1}B_0^T \partial_{x_0} V_0\notag\\
& =2( [{\bf P}^\dag_{\kappa,21},{\bf P}^\dag_{\kappa,22},{\bf P}^\dag_{\kappa,23}  ]\xi_k + {\bf s}^\dag_{\kappa, 2})^T (A_0 x_0  +F_0^\pi \bar z \\
 &\qquad\qquad- B_0 R_0^{-1}B_0^T
([{\bf P}^\dag_{0,11}, {\bf P}^\dag_{0,12}] \xi_0 +{\bf s}^\dag_{0,1}))\notag\\
&=2( [{\bf P}^\dag_{\kappa,21},{\bf P}^\dag_{\kappa,22},{\bf P}^\dag_{\kappa,23}  ]\xi_k + {\bf s}^\dag_{\kappa, 2})^T \cdot [( A_0- B_0 R_0^{-1}B_0^T{\bf P}^\dag_{0,11} )x_{0 }\\
&\qquad\qquad+ (F_0^\pi- B_0 R_0^{-1}B_0^T{\bf P}^\dag_{0,12})\bar z
 - B_0 R_0^{-1}B_0^T {\bf s}^\dag_{0,1}],\\
  \chi_3+\chi_7&=\mbox{Tr} ({\bf P}^\dag_{\kappa,22} D_0D_0+{\bf P}^\dag_{\kappa,11}DD),
\end{align*}
and
\begin{align*}
\chi_4&=\partial_{z_\kappa}^T V_\kappa(A_{\kappa} z_\kappa+Gx_0+F^\pi \bar z) \nonumber\\
&=2([{\bf P}^\dag_{\kappa,11},{\bf P}^\dag_{\kappa,12},{\bf P}^\dag_{\kappa,13}  ]\xi_\kappa+  {\bf s}^\dag_{\kappa,1})^T (A_{\kappa} z_\kappa+Gx_0+F^\pi \bar z),\\
\chi_5&=|[{\bf P}^\dag_{\kappa,11},{\bf P}^\dag_{\kappa,12},{\bf P}^\dag_{\kappa,13}  ]\xi_\kappa+  {\bf s}^\dag_{\kappa,1}|^2_{BR^{-1} B},
\\
\chi_6 &= \xi_\kappa^T[I, -\Gamma_1, -\Gamma_2^\pi ]^TQ[I, -\Gamma_1, -\Gamma_2^\pi ]\xi_\kappa  \nonumber \\
&\quad-2\eta^T Q [I, -\Gamma_1, -\Gamma_2^\pi ]\xi_\kappa+\eta^T Q\eta.
\end{align*}

Finally,
\begin{align}
\chi_{8}-\chi_{9} &= [\partial_{y_1}^T\partial_{\mu_1} V_\kappa, \cdots, \partial_{y_K}^T\partial_{\mu_K} V_\kappa]([\overline {\bf G}^\dag, \overline {\bf A}^\dag ]\xi_0 +\overline {\bf m}^\dag) \nonumber \\
&=2(z_\kappa^T {\bf P}^\dag_{\kappa, 13} + x_0^T{\bf P}^\dag_{\kappa, 23}+  \bar z^T{\bf P}^\dag_{\kappa, 33} +{\bf s}^{\dag T}_{\kappa,3}  )([\overline {\bf G}^\dag, \overline {\bf A}^\dag ]\xi_0 +\overline {\bf m}^\dag),\notag
\end{align}
and by Lemma \ref{lemma:dmuyy},
$\chi_{10}=0$.

By the above calculations,
the right hand sides of \eqref{ME1}--\eqref{ME2} may be written as
\begin{align}
&\chi_0=\xi_0^T \Theta_{0}(t) \xi_0+ 2\xi_0^T\theta_{0,1}(t)+\theta_{0,2}(t), \label{chi0T}  \\
&\chi= \xi_\kappa^T \Theta_{\kappa}(t) \xi_\kappa+ 2\xi_\kappa^T\theta_{\kappa, 1}(t)+\theta_{\kappa, 2}(t),\label{chikT}
\end{align}
where
\begin{align}
&\Theta_0={\bf P}_0^\dag
\begin{bmatrix}
A_0 &F_0^\pi\\
\overline {\bf G}^\dag & \overline {\bf A}^\dag
\end{bmatrix}
+ \begin{bmatrix}
A_0 &F_0^\pi\\
\overline {\bf G}^\dag & \overline {\bf A}^\dag
\end{bmatrix}^T {\bf P}_0^\dag -{\bf P}_0^\dag {\mathbb B}_0R_0^{-1}{\mathbb B}_0^T {\bf P}_0^\dag+{\mathbb Q}_0^\pi, \notag  \\  
&\theta_{0,1}=
\begin{bmatrix}
A_0 & F_0^\pi \\
\overline {\bf G}^\dag & \overline {\bf A}^\dag
\end{bmatrix}^T {\bf s}_0^\dag
-\begin{bmatrix}
{\bf P}_{0,11}^\dag \\
{\bf P}_{0,21}^\dag
\end{bmatrix}B_0 R_0^{-1} B_0^T {\bf s}_{0,1}^\dag
+\begin{bmatrix}
{\bf P}_{0,12}^\dag \\
{\bf P}_{0,22}^\dag
\end{bmatrix}\overline {\bf m}^\dag
 -  \bar \eta_0^\pi, \nonumber
\end{align}
and
\begin{align}
&\Theta_\kappa = {\bf P}^\dag_\kappa {\mathbb A}^\dag_\kappa  + {\mathbb A}_\kappa^{\dag T} {\bf P}^\dag_\kappa
 - {\bf P}^\dag_\kappa {\mathbb B}R^{-1}{\mathbb B}^T{\bf P}^\dag_\kappa+{\mathbb Q}^\pi,  \notag \\
&\theta_{\kappa,1} =({\mathbb A}_\kappa^{\dag T}  - {\bf P}^\dag_\kappa {\mathbb B}R^{-1}{\mathbb B}^T ){\bf s}^\dag_\kappa
  -
 \begin{bmatrix}
{\bf  P}^\dag_{\kappa,12}\\
{\bf  P}^\dag_{\kappa,22}\\
{\bf  P}^\dag_{\kappa,32}
 \end{bmatrix}\! B_0 R_0^{-1} B_0^T {\bf s}^\dag_{0,1}
 + \begin{bmatrix}
{\bf  P}^\dag_{\kappa,13}\\
{\bf  P}^\dag_{\kappa,23}\\
{\bf  P}^\dag_{\kappa,33}
 \end{bmatrix}\overline {\bf m}^\dag-  \bar \eta^\pi .\notag
\end{align}
The two terms $\theta_{0,2}(t)$ and $\theta_{\kappa,2}(t)$
can be determined but we omit the detail here.

Now if $(V_0, \cdots ,V_K)$ in \eqref{qs0}--\eqref{qsk} is indeed a solution of the master equations, we necessarily have
\begin{align}
\rho {\bf P}^\dag_0-\dot{\bf P}^\dag_0 = \Theta_0, \qquad \rho {\bf P}^\dag_\kappa
-\dot{\bf P}^\dag_\kappa = \Theta_\kappa,\quad 1\le \kappa \le K.\notag
\end{align}
Hence, \eqref{qsP0} and \eqref{qsPk} hold with the corresponding terminal conditions. In addition, $({\bf s}^\dag_0, {\bf s}^\dag_1, \cdots, {\bf s}^\dag_K)$ satisfies \eqref{qss0}--\eqref{qssk} as its unique solution.

Conversely, if \eqref{qsP0}--\eqref{qsPk} has a solution $( {\bf P}^\dag_0,\cdots, {\bf P}^\dag_K)$, we further uniquely solve $({\bf s}^\dag_0, {\bf s}^\dag_1, \cdots, {\bf s}^\dag_K)$ from \eqref{qss0}--\eqref{qssk}, and next solve
\begin{align}
\rho {\bf r}^\dag_0-\dot{\bf r}^\dag_0 = \theta_{0,2}, \qquad \rho {\bf r}^\dag_\kappa-\dot{\bf r}^\dag_\kappa
 = \theta_{\kappa,2},\quad 1\le \kappa \le K,\notag
\end{align}
where ${\bf r}^\dag_0(T)= \eta_{0f}^T Q_{0f}\eta_{0f}$ and  $  {\bf r}^\dag_\kappa(T)=\eta_{f}^T Q_f\eta_{f}$.
By the relation \eqref{chi0T}--\eqref{chikT}, then $(V_0, \cdots, V_K)$ given by \eqref{qs0}--\eqref{qsk} is a solution to the master equations. This completes  the proof of the theorem.

\end{document}